\documentclass[11pt]{amsart}
\usepackage{graphicx}
\usepackage{amssymb}

\usepackage{amssymb,amsmath,amsthm}
\usepackage[all]{xy}
 \usepackage{xy}
 \xyoption{curve}
 \xyoption{color}
 \xyoption{line}
\xyoption{arc}

\usepackage{url}
  \usepackage{color}
\usepackage{graphicx}
\usepackage[
            pdftex,
           colorlinks=true,
            urlcolor=blue,       
            filecolor=blue,     
            linkcolor=blue,       
            citecolor=blue,         
            pdftitle= {Title},
            pdfauthor={Author},
            pdfsubject={Subject},
            pdfkeywords={Key}
            pagebackref,
            pdfpagemode=None,
            bookmarksopen=true]
            {hyperref}
\usepackage[author=,english,status=draft]{fixme}
\usepackage{ifdraft}
\usepackage{bbm}
\usepackage{minitoc}
\ifdraft{
\usepackage{fancyhdr}
\pagestyle{fancy}
\fancyhead{}
\fancyfoot{}
\fancyhead[CO,CE]{Draft --- \today}
\fancyhead[RO, LE] {\thepage}
}

\usepackage{enumerate}

\date{}
\dedicatory{To Gopal Prasad}

\usepackage{amsthm}

\newtheorem{theorem}{Theorem}[section]
\newtheorem{lemma}[theorem]{Lemma}
\newtheorem{corollary}[theorem]{Corollary}
\newtheorem{proposition}[theorem]{Proposition}
\newtheorem{conjecture}[theorem]{Conjecture}

\theoremstyle{definition}
\newtheorem{definition}[theorem]{Definition}
\newtheorem{example}[theorem]{Example}

\theoremstyle{remark}
\newtheorem{remark}[theorem]{Remark}

\numberwithin{equation}{section}

\newcommand{\calA}{\mathcal{A}}
\newcommand{\calB}{\mathcal{B}}

\newcommand{\calE}{\mathcal{E}}
\newcommand{\calF}{\mathcal{F}}

\newcommand{\calI}{\mathcal{I}}

\newcommand{\calL}{\mathcal{L}}

\newcommand{\calN}{\mathcal{N}}

\newcommand{\calO}{\mathcal{O}}

\newcommand{\calP}{\mathcal{P}}

\renewcommand{\epsilon}{\varepsilon}

\newcommand{\frakm}{\mathfrak{m}}

\newcommand{\frakc}{\mathfrak{c}}

\newcommand{\frakf}{\mathfrak{f}}

\newcommand{\frake}{\mathfrak{e}}

\newcommand{\frakC}{\mathfrak{C}}
\newcommand{\frakV}{\mathfrak{V}}

\newcommand{\sfO}{\mathsf{O}}

\newcommand{\sfU}{\mathsf{U}}

\newcommand{\bbA}{\mathbb{A}}

\newcommand{\bmu}{\boldsymbol{\mu}}

\newcommand{\bbF}{\mathbb{F}}

\newcommand{\bbP}{\mathbb{P}}

\newcommand{\bbV}{\mathbb{V}}
\newcommand{\bbZ}{\mathbb{Z}}
\newcommand{\bbG}{\mathbb{G}}

\newcommand{\bfA}{\mathbf{A}}
\newcommand{\bfF}{\mathbf{F}}
\newcommand{\bfU}{\mathbf{U}}
\newcommand{\bfe}{\mathbf{e}}
\newcommand{\bfV}{\mathbf{V}}

\newcommand{\bfG}{\mathbf{G}}

\newcommand{\bfpic}{\mathbf{Pic}}

\newcommand{\sep}{\textup{sep}}
\newcommand{\fib}{\textup{fib}}

\newcommand{\half}{\frac{1}{2}}

\newcommand{\Aut}{\textup{Aut}}
\newcommand{\Hom}{\textup{Hom}}
\newcommand{\GL}{\textup{GL}}

\newcommand{\im}{\textup{Im}}

\newcommand{\Pic}{\textup{Pic}}

\newcommand{\Ker}{\textup{Ker}}

\newcommand{\Spec}{\textup{Spec}}

\newcommand{\et}{\textup{et}}

\newcommand{\loc}{\operatorname{loc}}
\newcommand{\Proj}{\operatorname{Proj}}

\newcommand{\Br}{\operatorname{Br}}

\newcommand{\fl}{\operatorname{fl}}

\newcommand{\Sym}{\textup{Sym}}

\newcommand{\Tors}{\textup{Tors}}

\newcommand{\Der}{\textup{Der}}

\newcommand{\Lie}{\textup{Lie}}
\newcommand{\NS}{\textup{NS}}

\usepackage[T2A,T1]{fontenc}
\newcommand{\Sh}{\mbox{\usefont{T2A}{\rmdefault}{m}{n}\CYRSH}}

\newcommand{\id}{\textup{id}}

\newcommand{\Jac}{\textup{Jac}}
\newcommand{\bfJ}{\mathbf{J}}

\newcommand{\ga}{\mathbb{G}_{a,F}}

\newcommand{\pp}{\textrm{pp}}

\renewcommand{\bar}[1]{{\overline{#1}}}

\newcommand{\Gal}{\operatorname{Gal}\nolimits}

\renewcommand\emptyset\varnothing

\newcommand{\beq}{\begin{equation}}
\newcommand{\eeq}{\end{equation}}

\begin{document}
\title[Torsors]{Integral models and  torsors of inseparable forms of $\bbG_a$}

\author[Igor Dolgachev]{Igor Dolgachev}
\address{Department of Mathematics, University of Michigan, 525 E. University Av., Ann Arbor, Mi, 49109, USA}
\email{idolga@umich.edu}

\begin{abstract} 
After recalling some basic facts about  $F$-wound commutative unipotent algebraic groups over 
an imperfect field $F$ we 
study their regular integral models over  Dedekind schemes of positive characteristic and compute 
the group of isomorphisms classes of torsors of one-dimensional groups.
\end{abstract}

\maketitle
\setcounter{tocdepth}{1}
\setcounter{section}{-1}

\section{Introduction}\label{S1}  

In this paper we study integral models of unipotent commutative algebraic groups over 
the field $K$ of rational function of a Dedekind scheme $S$ over an algebraically closed 
field $\Bbbk$ of positive characteristic. 
The interest to  unipotent groups over imperfect fields  is related to the study of 
pseudo-reductive algebraic groups \cite{CGP} 
 and  quasi-elliptic fibrations on smooth projective surfaces \cite{BM}, \cite{CDL}. 
The general theory of such groups   was started by J. Tits \cite{Tits}. 
Its earlier appearance in algebraic geometry goes back to 
the  work of 
M. Rosenlicht \cite{Rosenlicht}.  
A one-dimensional unipotent group over $K$ not isomorphic to the additive group $\bbG_{a,K}$  
admits a regular (but not smooth)  compactification, and, via the theory of relative minimal models 
of two-dimensional schemes, can be realized as the generic fiber of a 
flat morphism $f:X\to S$ of regular schemes. In the special case when the genus of a regular compactification 
is equal to one and $S$ is a proper smooth curve over $\Bbbk$, this leads to quasi-elliptic algebraic 
surfaces that play an important role in the 
classification of algebraic  
surfaces over fields of positive characteristic. Their appearance  in the theory of Enriques and K3 surfaces 
in  small characteristic  is the main motivation for writing this paper. 

A most notorious example of a commutative unipotent algebraic group over a field $F$ is a 
vector group, a 
finite-dimensional vector space over $F$ equipped with a structure of an algebraic group over $F$ isomorphic to a 
direct sum of the additive group $\bbG_{a}$ over $F$. Over any base scheme $S$, one can consider its 
form in Zariski topology, a vector group scheme $\bbV(\calE)$ over
 a base scheme $S$ associated to a locally free sheaf 
$\calE$ of $\calO_S$-modules. Its value on any scheme $a:T\to S$ is the 
group of global sections
of the dual of the sheaf $a^*\calE$.  One can also consider forms of $\bbV(\calE)$ in  \'etale or flat 
Grothendieck 
topology. It is known that  there are 
no non-trivial forms of a vector 
group in \'etale topology, however there are forms in flat topology if $F$ is not a perfect field. 
A non-trivial form of  
$\bbG_{a,F}$  is an example of a 
 \emph{$F$-wound unipotent group} as defined by J. Tits \cite{Tits}: a unipotent  algebraic group which 
does not contain any subgroups isomorphic to $\bbG_{a,F}$. A commutative connected unipotent group $\sfU_K$ 
of dimension $r$ killed by multiplication by the characteristic $p$ of $F$ can be embedded in a vector group 
$\bbG_{a,F}^{r+1}$ as the 
zero locus of a $p$-polynomial $\Phi$ in $r+1$ variables, a sum of monomials   in one variable of degree power of $p$.
After adjoining to the field some $p^k$-power roots of the coefficients, one obtains a group 
isomorphic to a vector group, so it becomes a purely inseparable form of a vector group. 
In the case when $\sfU_K$ is one-dimensional one can choose $\Phi$ to be a Russell polynomial, a $p$-polynomial 
of the form
$\Phi = u^{p^n}+v+a_1v^p+\cdots +a_mv^{p^m}$. The minimal extension $F'/F$ such 
that the group becomes isomorphic to $\bbG_{a,F'}$ is equal to $F(a_1^{1/p^n},\ldots,a_m^{1/p^n})$. 
By giving appropriate weights to the variables $u,v$ one can homogenize the affine curve given by equation $\Phi(u,v) = 0$ in 
$\bbA_F^2$ to obtain a weighted homogenous compactification $\bar{\sfU}$ of 
degree $p^{\max\{m,n\}}$ in the weighted homogenous 
projective plane $\bbP(1,1,p^{ \min\{m,n\}})$. 
If the polynomial $u^{p^n}-a_mv^{p^m}$ is irreducible in $F$, the
compactification $\bar{\sfU}$  is a regular projective curve of arithmetic genus 
$\half (p^{\min\{m,n\}}-1)(p^{\max\{m,n\}}-2)$. If this is not the case,  one takes a normalization of this curve to arrive at a 
regular compactification of smaller arithmetic genus. 

We will be mostly interested in the case when the field $F = K$. In this case one can study 
integral models $\bfU$ of one-dimensional $K$-wound unipotent group $\sfU_K$ over schemes $S$ 
obtained by localizations or completions of a smooth 
algebraic curve $C$ 
over $\Bbbk$. The theory of relative minimal models allows us to realize $\bfU$  as an 
open subscheme  
in a regular projective scheme $X\to S$ whose generic fiber $X_K$ is a regular compactification of a 
$K$-wound unipotent  group $\sfU_K$ over $K$ of some genus $g$. If $g > 0$, the group scheme $\bfU$ is the N\'eron model of 
$\sfU_K$ and it is also isomorphic to a closed 
group subscheme of the N\'eron model $\bfJ$ of the Jacobian variety $\Jac(X_K)$ of $X_K$ 
\cite[9.5,Theorem 4]{Bosch}. 
We say that a $K$-wound unipotent group $\sfU_K$ is of genus $g$ if the arithmetic genus of $X_K$ is equal to $g$.  
 In the case $g = 1$, $\sfU_K$ coincides 
with its Jacobian variety and an integral model $X\to S$ is an example of a jacobian quasi-elliptic 
fibration on a surface. The surfaces admitting such fibration were  first studied by Bombieri and Mumford in \cite{BM}. 

The theory of N\'eron models of elliptic curves and, more generally, abelian varieties, 
has an old history and was 
extensively  studied in algebraic geometry and number theory. It is an important tool in the study of 
principal homogeneous spaces (torsors) of 
abelian varieties. The isomorphism classes of  
torsors of an abelian variety $A$  over $K$ form a commutative group, 
 the \emph{Weil-Ch\^{a}telet group} $\textrm{WC}(A,K)$ isomorphic to the group of Galois cohomology  $H^1(K,A)$ 
 identified with the group of \'etale cohomology over $\Spec(K)$ of the abelian sheaf represented by $A$.
 The usual local-to-global study of this group splits into three parts: the study of the group of torsors 
 over a local field $K_v$ obtained from $K$ by a formal completion with respect to 
 a discrete valuation $v$ of $K$, the study of the subgroup of $\textrm{WC}(A,K)$ of isomorphism classes 
 of torsors which are triviial 
 for all valuations $v$ of $K$, the \emph{Tate-Shafarevich group}  
 $\Sh(A,K)$, and finally the study of the group of obstructions for reconstructing the isomorphism class of a torsor from 
 its localizations. The group $\Sh(A,K)$ admits a realization as the group of \'etale cohomology 
 $H_{\et}^1(C,\bfA)$, where $\bfA$ is the N\'eron model of $A$ over  $C$ considered as an 
 abelian sheaf in \'etale topology. The group of obstructions is the group of \'etale cohomology 
 $H_{\et}^2(C,\bfA)$. The existence of N\'eron models of commutative $K$-wound unipotent groups allows one to use 
 the same approach to the study of their torsors over $K$. In the case of one-dimensional 
 groups over fields $K$ of characteristic  $p = 3$
  this study was 
 first initiated by W. Lang \cite{Lang} and has been extended to characteristic 2 in \cite{CDL}. In this paper we 
 treat the case of arbitrary genus $g$. 
 
 Here is the content of the paper.
  
 In Section 1 we remind some known general facts about commutative unipotent groups over imperfect 
 fields. Here we introduce their weighted homogenous compactifications and compute their canonical sheaves.
 
 In Section 2, we specialize to the case of one-dimensional groups and 
 introduce the notion of a unipotent algebraic group $\sfU$ of genus $g$. 
 We discuss a known classification 
 of Russell equations defining groups of genus 0 and 1. 
 As we have already mentioned earlier,  the Jacobian variety of a 
 regular compactification of a genus $g> 0$ curve is a commutative $F$-wound unipotent group of dimension $g$. Although 
 $\sfU$ is a $p$-torsion group, the group $\Jac(X)$ is a $p^s$-torsion, where $s$ is the 
 smallest positive 
 integer such that $\sfU_K$ splits over $K^{1/p^s}$ or become of genus $0$ \cite{Achet}. 
An interesting question for which we do not have  an answer is 
 whether $\sfU$ coincides with the $p$-torsion part of $\Jac(X)$.
 
 In Section 3 we study minimal models of a unipotent group of genus $g$  over a Dedekind scheme over a field of characteristic 
 $p > 0$. We give an explicit description of such a model over a global base as a closed subscheme of 
 a group scheme locally isomorphic to a vector group scheme over the base.
 
 In Section 4, we study integral models of regular compactifications of
  rational $K$-wound one-dimensional unipotent groups over the field $K$ of rational 
  functions of a smooth projective curve $C$. They are 
 isomorphic to minimal ruled surfaces $X\to  C$ with an inseparable  bisection as the boundary of an integral 
 model of the group.  
 For example, if $C\cong \bbP^1$, then $X \cong \bfF_1$ and  the 
 boundary is the pre-image of the strange conic in the plane with the pole at the center of the blow-up. 

 In Section 5 we begin the study of torsors of 
 unipotent groups of genus $g$ by considering the local case where  
 the ground field is the field of formal power series over an algebraically closed field 
 $\Bbbk$ of characteristic $p$.
 The Weil-Ch\^{a}telet group in this case is the quotient group $K/\Phi(K\oplus K)$, where $\Phi$ is the 
 $p$-polynomial defining the group. In spite of this simple realization the group is very difficult to 
 compute and we succeeded to do this only in the case $g=0$, where it is trivial and in the case 
 $g = 1$, where it was computed in the case $p = 3$ by W. Lang. An integral model $X\to \Spec(R)$  of this group over 
 the ring of integers $R = \Bbbk[[t]]$ is a genus one fibration with the closed fiber of multiplicity $p$. 
 The  important invariant of this fibration is the length $l(T)$ of the torsion part $T$ of the 
 $R$-module $H^1(X,\calO_X)$. A torsor is tame (resp. wild) when $T = 0$ (resp. $T\ne 0$). An open problem 
 which we have been unable to solve is to compute this invariant in terms of the Russell equation of the model 
 and compare it with the computation of this invariant in terms of the infinitesimal neighborhoods of the multiple fiber  
 given by Michel Raynaud in his unpublished work (see \cite[4.2]{CDL}).

 In Section 6 we treat the case of global fields, we compute the Tate-Shafarevich group 
 $H_{\et}^1(C,\bfU)$ and the obstruction group $H_{\et}^2(C,\bfU)$ of the N\'eron model $\bfU$ of $\sfU_K$ over a complete smooth curve $C$ with the field of rational functions 
 $K$.  We also discuss here a mysterious relationship between the rank of the elementary 
 abelian group of global sections $\bfU(C)$ and the $p$-rank of some hyperelliptic curves in 
 characteristic $p > 2$.    
 
 I thank T. Katsura, G. Martin, and M. Sch\"utt  for helpful comments on the  paper. I am thankful to the referee foa thor

 \section{Commutative unipotent algebraic groups over non-perfect fields} \label{S2}
 Throughout  this paper $F$ denotes a   
 field of characteristic $p > 0$. In this section  we will review some basic facts about 
 commutative unipotent algebraic groups over 
 $F$ with emphasis on the case when $F$ is imperfect. 
 There are several sources to refer to, e.g.
  \cite[Appendix B]{CGP}, \cite{Kambayashi}, \cite{Oesterle},  \cite{Tits}. 
 Recall that a linear algebraic group $\sfU$ over a field $F$ is called \emph{unipotent} 
 if it is isomorphic to a subgroup of the algebraic group of  
  upper triangular matrices over $F$ with  1 at the diagonal. 
 We will be interested only in commutative unipotent algebraic groups. 
There is a classification of commutative connected unipotent algebraic groups over a perfect field  $F$ 
  \cite[Chapter 5,\S 1]{Demazure}. No classification is known in the case when $F$ is imperfect.
 
 Each commutative unipotent algebraic group over $F$ is a 
 $p^n$-torsion group. It admits a composition 
 series with $p$-torsion quotients. A $p$-torsion group is isomorphic to 
 a codimension one subgroup of  a vector group $\bbG_{a,F}^{r+1}$ defined as the zero locus of 
 a \emph{$p$-polynomial}
  \beq\label{p-polynomial}
\Phi(x_1,\ldots,x_{r+1})  = \sum_{i=1}^{r+1}(\sum_{j=0}^{k_i}c_{j}^{(i)}x_i^{p^j}), \quad c_{k_i}^{(i)}\ne 0.
 \eeq
 (see \cite[Proposition B1.1.13]{CGP}).
The polynomial $\Phi^{\pp}: = \sum_{i=1}^{r+1}c_{k_i}^{(i)}x_i^{p^{k_i}}$ is called 
the \emph{principal part} of $\Phi$. 
We assume  that at least one of the coefficients $c_{0}^{(1)},\ldots,c_{0}^{r+1}$ is not zero. 
If this condition is satisfied we say that $\Phi$ is a \emph{separable $p$-polynomial}.
This condition guarantees that 
the closed subscheme $V(\Phi)$ of $\bbA_F^{r+1}$ is smooth, and hence $\sfU$ is a smooth affine group scheme.

If, for example, $c_0^{(i)} \ne 0$,  then the restriction to $\sfU = V(\Phi)$ of the projection $\ga^{r+1}$ to $\ga^{r}$ given by 
 $(x_1,\ldots,x_{r+1})\to (x_1,\ldots,x_{i-1},x_{i+1},\ldots,x_{r+1})$ is an \'etale isogeny of 
 degree $p^{k_i}$.

\begin{proposition} Let $G$ be a connected subgroup of $\ga^{r+1}, r\ge 1,$ given by a separable $p$-polynomial 
$\Phi$. There exists a  radical extension $F'$ of $F$  such that 
$G_{F'}\cong \bbG_{a,F'}^r$,
\end{proposition} 

\begin{proof} Without loss of generality, we may assume that $c_0^{(1)} = 1$. Replacing $x_1$ with 
$x_1+c_0^{(2)}x_2+\cdots+c_0^{(r+1)}x_{r+1}$, we may write 
$\Phi = x_1+\Phi_1$, where $\Phi_1$ is a  $p$-polynomial with no monomials of degree $1$. 
Adding to $F$ all $p^{j}$-roots of the coefficients $c_j^{(i)}$ of $\Phi_1$, we obtain a finite purely inseparable  extension 
$F'$ of $F$ such that $\Phi = x_1+\Phi_2(x_1,\ldots,x_{r+1}) ^{p^s}$, where $\Phi_2$ is a separable 
$p$-polynomial over $F'$. Let $G\to \bbG_{a,F'}$ be the homomorphism of algebraic groups over $F'$ given by 
$(x_1,\ldots,x_{r+1}) \to \Phi_2(x_1,\ldots,x_{r+1}).$ Since $\Phi_2$ is separable, the homomorphism is 
surjective. Its kernel is isomorphic to 
$V(\Phi_2(0,x_2,\ldots,x_{r+1}))\subset \bbG_{a,F'}^{r}$. If $r=1$, its is an \'etale cover of $\bbG_a$, and hence 
is isomorphic to $\bbG_{a,F'}$. If $r > 1$, we apply induction on $r$, to obtain that $G$ is an extension of 
$\bbG_a^{r-1}$ and $\bbG_a$ (after replace $F'$ by some finite radical extension). Since $G$ is a connected 
$p$-torsion 
algebraic group, it must be isomorphic to a vector group $\bbG_{a}^r$ (see \cite[Lemma B1.10]{CGP}). 
\end{proof} 
 
We say that $G$ is split over an extension $F'$ if $G_{F'} \cong \bbG_{a,F'}$. Such $F'$ is called 
a \emph{splitting extension}.

It follows that  any smooth connected unipotent $p$-torsion algebraic group is a purely inseparable  
form of $\bbG_a^r$. 
  There are no non-trivial 
 separable forms \cite[\S 2]{Kambayashi}. Note that for $r > 1$ this fact does not follow from the 
 usual Galois cohomology argument because $\Aut(\bbG_a^r)$ is strictly larger than 
 $\GL_F(r)$.
 
It is well known that over a perfect field $F$ every connected 
commutative unipotent group is isomorphic, as a $F$-scheme,  
to  affine space $\bbA_F^n$. In particular, it admits a non-constant morphism of algebraic varieties 
$\bbA_F^1\to \sfU$.

 \begin{definition}  
 A smooth connected unipotent $F$-group $\sfU$ over $F$ is called \emph{$F$-wound} if 
 every $F$-morphism of $F$-schemes 
 $\bbA_{F}^1\to \sfU$ is constant. 
 \end{definition}

It is known that a smooth unipotent group is $F$-wound if and only if it does not contain 
closed subgroups isomorphic to $\bbG_{a,F}$ \cite[Theorem 4.3.1]{Kambayashi}, \cite[Corollary B.2.6]{CGP}.

 For example, any non-trivial inseparable form of $\bbG_{a,F}$ is $F$-wound. 

We refer for the proof of  the following lemma to \cite[Lemma B.1.7]{CGP}

\begin{lemma}\label{lem:conrad2} The group $\sfU$ given by the zero scheme of a
 $p$-polynomial $\Phi$ in $\ga^{r+1}$ is an $F$-wound group if 
the principal part $\Phi^{\pp}$ of $\Phi$ has no zeros in $F^{r+1}\setminus \{0\}$.
\end{lemma}

\begin{remark}\label{ext:wound} For any separable extension $F'/F$, the group $\sfU_{F'}$ is 
 an $F'$-wound group. 
However, it may be still  $F'$-wound after an inseparable extension $F'/F$. For example, if $\sfU$ is given by equation 
$u^2+v+av^2+bv^4 = 0,\  a^{1/2}\not\in F$, then it is $F$-wound and $F'$-wound under extension $F' = F(b^{1/4})$. 
Under this extension the group is isomorphic to a wound group given by equation 
$y^2+v+av^2 = 0$. 

Another remark is that the converse of  Lemma \ref{lem:conrad2} is wrong. For example, applying an automorphism 
$(x,y) \mapsto (x+x^p,y)$ to an $F$-wound group $y^p+x+ax^p = 0$ we obtain that the principal part has 
zero in $F^2\setminus \{0\}$.
\end{remark}

The following theorem is Corollary  B.2.5 from \cite{CGP}.
 
 \begin{theorem}\label{decomp} Every smooth connected unipotent $p$-torsion 
 commutative algebraic group over $F$ is a 
 direct product of $\sfU = V\times \sfU'$, where $V\cong \ga^s$ and $\sfU'$ is a 
 smooth connected $F$-wound unipotent group.
 \end{theorem}

 We know that $G=\ga^r$ admits the projective space $\bbP_F^r$ as a 
 $G$-equivariant regular projective compactification $\bar{G}$ of $G$. It is 
 minimal in the sense that $G$ acts trivially on the complement $\bar{G}\setminus G$.

 \begin{definition} Let $G$ be a connected commutative algebraic group over $F$. 
 A \emph{minimal  compactification} of  $G$ is a 
projective $F$-scheme $\bar{G}$  such that 
 $G$ admits a $G$-equivariant closed embedding in $\bar{G}$ and  
 $G$ acts trivially on the complement
 $\bar{G}\setminus G$.
 \end{definition}
 
The  following theorem can be found in \cite[10.2, Proposition 11]{Bosch}. 

 \begin{theorem}\label{regcomp} Every  smooth connected commutative unipotent 
 $p$-torsion algebraic group $\sfU$ over $F$ admits a minimal  compactification 
 $G\hookrightarrow \bar{G}$. 
 The group is $F$-wound if and only if the complement $\bar{\sfU}\setminus \sfU$ 
 does not have a rational point over $F$. 
 \end{theorem}

\begin{proof} By Theorem \ref{decomp}, we may assume that $\sfU$ is $F$-wound. 
We know that $\sfU$ is isomorphic to the zero scheme $V(\Phi)$ of a separable $p$-polynomial 
in $\ga^{r+1}$. Let 
$\Phi$ be as in equation \eqref{p-polynomial} with 
$$k_1\ge \cdots \ge k_{r+1}$$
and
\beq\label{regcompact}
\bar{\Phi} = \sum_{i=1}^{r+1}(\sum_{j=0}^{k_i}c_{i}^{(j)}t_0^{p^{k_i-j}}t_i^{p^j}),
\eeq
be a weighted homogenization of $\Phi$. We view $\bar{\Phi} = 0$ 
as a hypersurface $X = \bar{\sfU}$ of degree $p^r$ in the weighted projective space 
$$\bbP(q_0,q_1,\ldots,q_{r+1})= \bbP(1,1,p^{k_1-k_2},p^{k_1-k_3},\ldots,p^{k_1-k_{r+1}}).$$ 

We call such a compactification a \emph{weighted homogeneous compactification}.
 
It is immediate to see that the complement $\bar{\sfU}\setminus \sfU$ is equal to 
$V(\Phi^{\pp})$ in the hyperplane $V(t_0)$. The group $\sfU$ acts on $\bar{\sfU}$ by translations
$$(a_1,\ldots,a_{r+1}):(t_0,t_1,\ldots,t_{r+1})\mapsto 
(t_0,t_1+a_1t_0,\ldots,t_{r+1}+a_{r+1}t_{r+1}).$$
The action is identical on the complement and also, by Lemma \ref{lem:conrad1} the complement has 
no rational points over $F$ if and only if $\sfU$ is an $F$-wound group.

Note that $X = \bar{\sfU}$ is not smooth. The singular locus of $ X\otimes_F\bar{F}$ is the 
hyperplane section $V(t_0)$.
\end{proof}

\begin{proposition}\label{normality} The weighted homogenous compactification $X$ from above is normal over $F$ if and only if 
the principal part $\Phi^{\pp}$ of $\Phi$   is reduced over $F$.
\end{proposition}

\begin{proof} By Serre's criterion of normality, we have to check that  $X$ is regular at infinity. 
We know that the hyperplane $t_0 = 0$ cuts out the hypersurface along 
the hypersurface $V(\Phi^{\pp})$.

Let $P$ be the polynomial obtained from $\bar{\Phi}$ by dehomogenization 
with respect to the variable $t_1$, i.e. dividing by $t_1^{p^{k_1}}$ and using the new variables 
$z_1 = t_0/t_1, z_2 = t_2/t_1^{p^{k_1-k_2}},\ldots,z_{r} = t_{r+1}/t_1^{p^{k_1-k_{r+1}}}$.
 We can write  
 $P = \widetilde{\Phi}^{\pp}(1,z_2,\ldots,z_r)+z_1^pB(z_1,\ldots,z_r)$, where  $\tilde{\Phi}^{\pp}$ 
is the dehomogenization with respect to $t_1$ of the weighted homogeneous polynomial $\Phi^{\pp}$ of degree $p^{k_1}$ in variables 
$t_1,\ldots,t_{r+1}$.
The  quotient of the ring $F[z_1,\ldots,z_r]/(P)$ by the ideal $(z_1)$ is 
isomorphic to $F[z_2,\ldots,z_r]/(\tilde{\Phi}^{\pp})$. Thus the ideal $(z_1)$ is prime  
if and only if $\widetilde{\Phi}^{\pp}$ is irreducible over $F$, or, equivalently, $\Phi^{\pp}$ is irreducible over $F$.
So, $\Phi^{\pp}$ is irreducible over $F$ if and only if the boundary is defined by a principal prime ideal 
and hence if and only if $X$ is regular.
\end{proof}

Let $X = \bar{\sfU}$ be given as in the proof of the previous theorem. A well-known formula for the canonical sheaf 
of a hypersurface in a 
weighted projective space gives
\beq\label{canclass}
\omega_{X} \cong \calO_{X}(p^{k_1}-2-\sum_{i=2}^{r+1}p^{k_1-k_i}).
\eeq
For example, if $p = k_1=k_2 =k_3 = r = 2$, $X$ is a quartic surface in $\bbP^3$. 
If we choose the equation in order $X$ is normal and all its singular points are rational double points, 
then its minimal resolution of singularities is a K3 surface. 

\section{Inseparable forms of $\bbG_a$}\label{S3} 
Specializing to the case $r = 1$, we see that any $F$-wound one-dimensional 
unipotent group (automatically commutative) is a subgroup of $\ga^2$ given by an equation
$$u^{p^m}+a_{m-1}u^{p^{m-1}}+\cdots+ a_0u+b_nv^{p^n}+b_{n-1}v^{p^{n-1}}+\cdots+ b_0v = 0,$$
where $n\le m$ and $b_n^{p^{-n}}\not\in F$. 

In fact,  we have the following  theorem of Russell \cite{Russell}.
\begin{theorem}\label{russell} Each smooth connected one dimensional 
unipotent group $\sfU$ is isomorphic to a subgroup 
of $\bbG_{a,F}^2$ given by a $p$-polynomial of the form  
\beq\label{russell1}
\Phi(u,v) = u^{p^n}+v+a_1v^p\cdots+a_mv^{p^m} = 0, 
\eeq
\cite{Russell}. The smallest  purely inseparable  extension $F'/F$  that splits $\sfU$ 
 is  equal 
$F' = F(a_1^{p^{-1/n}},\ldots,a_m^{p^{-1/n}})$. 
\end{theorem}
The number $n$ is called the \emph{height} of $\sfU$.

The polynomial $\Phi(u,v)$ is obviously  not unique. 
For example, if $n\le m$, we can also add to $a_m$ any $p^n$-th power of an element from $K$ 
without changing the isomorphism class.
 Proposition 2.3 from loc.cit. gives  a precise relationship 
between two polynomials defining isomorphic groups. 

 We refer to \eqref{russell1} as 
a \emph{Russell equation}. The  projections 
$\pi_1:\sfU\to \ga, (u,v)\mapsto u,$ and $\pi_2:\sfU\to \ga, (u,v) \mapsto v,$ make the Cartesian square
$$\xymatrix{\sfU\ar[r]^{\pi_2}\ar[d]^{\pi_1}&\ga\ar[d]^\tau\\
\ga\ar[r]^{\bfF^r}&\ga},
$$
where $\bfF:\ga\to \ga$ is the Frobenius homomorphism and $\tau$ is given by the $p$-polynomial 
$v+v^p+\cdots+a_mv^{p^m}$.

\begin{itemize}
\item In this section $\sfU$ always denotes a smooth connected $F$-wound unipotent one-dimensional algebraic group over $F$.
\end{itemize}

\begin{proposition} Any  compactification $X$ of $\sfU$ is unibranched.
\end{proposition}

\begin{proof} It is rather obvious. Let $F'$ be a splitting field of $\sfU$, then
$X_{F'}$ is a compactification of $\bbG_{a,F'}$. Its normalization $\tilde{X}$ is isomorphic to 
$\bbP_{F'}^1$ and has one point as the complement. Thus $\tilde{X}$ and 
$X_{F'}$ are homeomorphic and since $F'/F$ is purely inseparable, $X_{F'}$ is 
homeomorphic to $X$ and has only one point at infinity.
\end{proof}

\begin{proposition} A regular compactification $X$ of $\sfU$ is a minimal compactification and the 
boundary consists of one point. 
\end{proposition}

\begin{proof} This immediately follows from the fact that the action of $\sfU$ on itself extends 
to an action of $\sfU$ on $X$ that acts trivially on the boundary \cite[Lemma 2.14]{Laurent}. 
Let $F'/F$ be a purely inseparable 
 splitting extension
of $\sfU$. Then $X_{F'}$ is a compactification of $\bbG_{a,F'}$ and its normalization is a 
regular compactification of $\bbG_{a,F'}$. It must be isomorphic to $\bbP_{F'}^1$ with a 
boundary consisting of one point. Thus $X_{F'}$ has one unibranched    singular point and 
$\bbP_{F'}^1$ is  homeomorphic to $X'$. Since $F'/F$ is purely inseparable, $X_{F'}$ is homeomorphic to $X$.
Thus the boundary of $X$  consists of one point.

\end{proof} 

Although a regular compactification of $\sfU$ is unique, there are many non-regular compactification.

Specializing equation \eqref{regcompact} to the case of a one-dimensional group, we get a 
compactification of $\sfU$ given as a curve of degree $p^{\max\{m,n\}}$ in the weighted homogeneous plane
$\bbP(1,1,p^{\max\{m,n\}-\min\{m,n\}})$ given by equation
\beq\label{whcompnm}
t_2^{p^n}+t_0^{p^m}t_1+\cdots+a_{m-1}t_0^{p}t_1^{p^{m-1}}+a_mt_1^{p^m} = 0
\eeq
if $n\le m$ and 
\beq\label{whcompmn}t_2^{p^n}+t_0^{p^m-1}t_1+\cdots+a_{m-1}t_0^{p^m-p^{m-1}}t_1^{p^{m-1}}+a_mt_1^{p^m} = 0
\eeq
if $m\le n$. 
 
It contains $\sfU$ as the complement of the hyperplane at infinity $V(t_0)$. 
 
Specializing formula \eqref{canclass}, we obtain
$$\omega_X \cong \calO_X(-2-p^{\max\{m,n\}-\min\{m,n\}}+p^{\max\{m,n\}})
$$
and 
\beq\label{unigenus}
p_a(X) = \half (p^{\min\{m,n\}}-1)(p^{\max\{m,n\}}-2).
\eeq 
(see also \cite[Theorem 1.25]{Achet}).

The curve has a unique unibranched non-smooth point  with coordinates $(x_0,y_0,0)\in X(\bar{F}) $, where 
$y_0^{p^n}+a_mx_0^{p^m} = 0$. It is regular if and only if $u^{p^n}+a_mv^{p^m}$ is reduced over $F$, equivalently,
$a_m$ is not a $p$th-power in $F$. The residue field of the boundary point is $F(a_m^{1/p^n})$. It is a 
subextension of the splitting field of $\sfU$.

 Let 
$q:\tilde{X}\to X$ be the normalization of $X$ over $F$. Recall that the annihilator ideal $\calI$ of 
$q_*\calO_{\tilde{X}}/\calO_X$ can be considered as an ideal sheaf on $\tilde{X}$. The closed 
subscheme $\frakc$ 
of $\tilde{X}$ defined by this ideal is called the \emph{conductor} of the normalization (the same name applies to the 
closed subscheme of $X$ defined by the ideal $\calI$). A curve $X$ is \emph{unibranched} if the support of 
$\frakc$ over each singular point consists of one point. It is \emph{semi-normal} if  $\frakc$  is reduced 
(e.g. has only ordinary double points as its singularities).

It is easy to see that a semi-normal unibranched curve over an algebraically closed field is regular.

Let $X$ be a regular compactification of $\sfU$ and $P_\infty$ its boundary point. One can 
construct a semi-normal unibranched compactification $X'$ by ``pinching'' the boundary point. This means that 
$X$ is a normalization of $X'$ and the conductor is equal to $P_\infty$ (see \cite{Ferrand}). 
The residue field of the boundary point of $X'$ is an  extension of $F$ that is contained in the splitting 
extension of $\sfU$. Conversely, each such extension can be realized as the residue 
field of the boundary point of some semi-normal compactification of $\sfU$ \cite{Laurent}.


\begin{definition}  $\sfU$ is 
said to be of arithmetic genus $g$ if it admits a 
minimal regular compactification of arithmetic genus $g$. We say that 
$\sfU$  is 
 \emph{quasi-rational} (resp. \emph{quasi-elliptic}) if its  arithmetic genus 
 is equal to 
  $0$ (resp. $1$). 
\end{definition}

Since a regular model is unique, up to isomorphism,  
the arithmetic genus is a birational invariant of a one-dimensional 
$F$-wound unipotent group.

The genus formula \eqref{unigenus} suggests that there must be a bound on $p$ for a fixed $g$. 
In fact, it is known that the arithmetical genus of a regular curve can drop under an 
inseparable extension of the ground field  only if 
$$p\le 2g+1$$
\cite{Tate} (see also a nice proof in \cite[Lemma 9]{Nick}).
For example, genus 1 unipotent group may exist only if $p = 2,3$. 

The following Proposition was first proven by Rosenlicht \cite{Rosenlicht}.
\begin{proposition}
A $F$-wound one-dimensional unipotent group is quasi-rational if and only if it is isomorphic to a subgroup of 
$\bbG_{a,F}^2$ given by equation 
\beq\label{ratunipgroup}
u^2+v+av^2 = 0,
\eeq
where $a$ is not a square in $F$. In particular, a quasi-rational $F$-wound  unipotent group exists only if 
$p =2$.
\end{proposition}

\begin{proof} Any regular curve of genus $0$ over $F$ that has a rational point in $F$ is 
isomorphic to $\bbP_F^1$. We know that the action of 
$\sfU$ on itself extends to an action on $\bbP^1$ fixing the point 
$P_\infty = \bbP^1\setminus \sfU$. We assume that the image of $0\in \sfU$ has homogenous coordinates 
$(t_0:t_1)$ equal to $(0:1)$. Thus there is a 
non-trivial transformation $T:(t_0:t_1)\to (\alpha t_1+\beta t_0:\gamma t_1+\delta t_0)$ that fixes $P_{\infty}$. 
But the residue field of any fixed point  of $T$
 is of degree $\le 2$ over $F$. 
 Since $\sfU$ is $F$-wound, $P_\infty$ is not a rational point (in this case the complement is $\bbA_F^1$). Thus 
 $\sfU$ is the complement of a point of degree $2$. Since it is inseparable, $p = 2$. The linear system 
 $|P_{\infty}|$ maps $\bbP_F^1$ isomorphically to a conic in $\bbP_F^2$ and there exists a line  
 defined over $F$  
that intersects the conic at one point equal to the image of $P_{\infty}$. 
We choose projective coordinates $(t_0:t_1:t_2)$ in the plane such that the equation of the line is 
$t_0 = 0$ and the image of $P_\infty$ is the point $(0:0:1)$. It is easy to see that the  equation
of  the conic can be reduced to the form
$t_1^2+t_0t_2+at_2^2 = 0$. In affine coordinates $u = t_1/t_0,v= t_2/t_0$, we obtain the asserted equation of $\sfU$.

\end{proof}
We can also find an explicit  formula for the group law on a quasi-rational unipotent group. 
Let $(u,v) = (\frac{s}{1+as^2}, \frac{s^2}{1+as^2})$ be a rational parameterization of $\sfU$. 
Then the group law is given in terms of the parameter $s$ by the formula
\beq\label{grouplaw}
s_1\oplus s_2 = \frac{s_1+s_2}{1+as_1s_2}.
\eeq
The group is isomorphic to $\bbG_a$ only if $a$ is a square in $F$. 

Let us now find equations of quasi-elliptic curves. The following proposition is due to C. Queen \cite{Queen}.

\begin{proposition}\label{queen} A unipotent group is quasi-elliptic if and only if its Russell equation is one of the following.
\begin{enumerate}
\item $p = 2$
\begin{itemize}
\item[(a)] $u^2+v+a_1v^2+a_2v^4 = 0,\  a_2^{1/2}\not\in F, [F(a_1^{1/2},a_2^{1/2}):F] = 2$;
\item[(b)]  $u^4+v+av^2 =  0,\  [F(a^{1/2}):F] = 2$;
\item[(c)] $u^4+v+av^2+b^2v^4 = 0, \ [F(a^{1/2},b^{1/2}):F] = 4$.
\end{itemize}
\item $p = 3$
$$u^3+v+av^3  = 0,\  [F(a^{1/3}):F]= 3.$$
\end{enumerate} 
\end{proposition}

\begin{proof} All these equations are Russell equations, so we can use 
 the weighted homogeneous compactification to check the genus. The genus formula gives that 
 $p_a(X)$ is indeed equal to 1 except in case 1(c), where it is equal to 3. Since $a_m$ is not a 
 $p$-th power in case $p_a(X) = 1$, the compactification is regular, hence the genus is 1. In case 1(c), 
 $a_m$ is a $p$th power, and the compactification is not regular. We use 
 its equation from \eqref{whcompnm} with $n = m = 4$. In affine coordinates $x = t_0/t_2.y = t_1/t_2$, the equation is 
 $y^4+x^3+a_1x^2+c^2 = 0$. Let $t = \frac{y^2+c}{x}.$ Then $t^2 = x+a_1$ shows that $t$ belongs to the 
 integral closure of the coordinate ring  of the affine curve in its fraction field. 
 Substituting $x = t^2+a_1$ in the equation, we obtain the equation 
 $y^2+t^3+a_1t+c = 0$. This shows that the normalization of $X$ is a regular plane cubic curve of 
 arithmetic genus 1 with equation 
 $$t_0t_2^2+t_1^3+at_1t_0^2+ct_0^3 = 0.$$
  It contains $\sfU$ as an open subset $D(t_1^2+a_1)$. The residue field of the boundary is 
  $F(a_1^{1/2},c^{1/2})$. If   
  $a$ is  a square, the change $u\mapsto u+av+bv^2$ transforms the equation to 
  $u^2+v = 0$, so the genus is zero. If $b^{1/2}\in F(a^{1/2})$, then $b = a\alpha^2+\beta^2$ for some
  $\alpha,\beta\in F$. Replacing $y$ with $y+\beta$, we may assume that $\beta = 0$. Then setting 
  $v' = \frac{a}{x^2+a}, u' = \frac{x+\alpha^2}{y+\alpha x}$, we obtain an equation 
  $v'{}^2+v'+au'{}^4$ which is reduced to case 1(b) after multiplyting it 
  by $a^{-1}$ and replacing  $v'$ with $a^{-1}v'$. 
  
  We refer to \cite{Queen} for proving that starting from a Weierstrass equation 
  $t_2^2t_0+t_1^3+at_0^2t_1^2+bt_0^3  = 0$ of a regular genus one curve, if $p = 2$, one can reduce it 
  to the case 1(a)    
  and $a$ is a square, to case 1(b) if $a$ is not square but $[F(a^{1/2},b^{1/2}):F] = 2$ and to case 1(c) if 
  $[F(a^{1/2},b^{1/2}):F] = 4.$ If $p = 3$, then $a = 0$  and, using affine coordinates 
  $u = t_0/t_2, v= t_1/t_2$, we obtain case 2.
  \end{proof}

 \begin{remark} 
 The notion of genus is not conserved under inseparable extensions. 
 In case 1 (c) from the previous Proposition, the genus of $\bar{\sfU}$ is 
 equal to one. 
 However, after we adjoin $a_2^{1/2}$, the group becomes of 
 genus $0$.
 \end{remark}

Let $X$ be a regular proper geometrically reduced and geometrically irreducible curve of genus $g > 0$ over a field $F$. 
The (generalized) Jacobian $\Jac(X)$ is the connected component of the identity 
$\mathbf{Pic}_{X/F}^\circ$ of the Picard scheme $\mathbf{Pic}_{X/F}$ of $X$. It is a connected 
commutative algebraic group of dimension $g$. It is an abelian variety  if and only if $X$ is smooth. 
If $X$ is not smooth, then $\Jac(X)$ does not contain neither $\bbG_{a,F}$ nor $\bbG_{m,F}$ 
\cite[9.2, Proposition 4]{Bosch}. If $X_{\bar{F}}$ is a unibranched rational curve, $\Jac(X)$ is an $F$-wound 
unipotent group \cite[4.1]{Achet}. If the geometric genus of $X_{\bar{F}}$ is positive it could be a pseudo-abelian variety 
\cite[Example 3.1]{Totaro}.

In particular, if  $X = \bar{\sfU}$ is a regular compactification of an $F$-wound unipotent group $\sfU$  of genus $g$, 
then $\Jac(X)$ is an $F$-wound unipotent group of dimension $g$ that contains $\sfU$ \cite{Achet}.

The  
degree homomorphism 
$\deg:\bfpic_{X/F}\to (\bbZ)_F$ of group schemes over $F$ is surjective because $X$ has a rational point. 
The restriction homomorphism $r:\bfpic_{X/F}\to \bfpic_{\sfU/F}$ is also surjective because $X$ is regular.
The kernel of $r$ is an infinite cyclic group generated by the invertible sheaf $\calO_X(P_\infty)$. 
Its image under the homomorphism $\deg$ is a cyclic subgroup of $(\bbZ)_F$ generated by $p^k$, where 
$p^k = \deg(P_\infty)$  equal to the minimal degree of the splitting 
extension of $\sfU$. This leads to an exact sequence
$$0\to \bfpic_{X/F}^0\to \bfpic_{\sfU/F} \to (\bbZ/p^k\bbZ)_F\to 0,$$
of algebraic groups \cite[Theorem 6.10.1]{Kambayashi}.

Let $n'$ be the smallest positive integer such that $\sfU_{F^{-p^{n'}}}$ is isomorphic to $\bbG_a$ or a 
unipotent group of genus $0$. 
We call it the \emph{pre-height} of $\sfU$. Obviously, $n'\le n$

The following theorem is proved in \cite[Theorem 4.4]{Achet}.

\begin{theorem}\label{thm:ptorsion} $\Jac(X) = \bfpic^0(X)$ is a $p^{n'}$-torsion $F$-wound unipotent group, where
$n$ is the pre-height of $\sfU$. 
\end{theorem}

Note that the theorem agrees with our classification of  quasi-elliptic unipotent groups. 
Indeed, in this case the pre-height $n'$ is always equal to $1$.  

Recall from \cite[Definition B.3.1]{CGP} that the \emph{cckp-kernel} of a smooth connected unipotent group 
over $F$ is its maximal smooth connected $p$-torsion central subgroup. We do not know whether 
the cckp-kernel of $\Jac(X)$ coincides with $\sfU$.

\section{Integral models over a Dedekind scheme}\label{S4}
 From now on,  $C$ will denote a  
Dedekind scheme over an algebraically closed  field $\Bbbk$ of characteristic $p > 0$ obtained by 
localization, henselization or completion from a complete smooth curve $C$ over $\Bbbk$. 
Let $K$ denote the field of rational functions of $C$, i.e. 
the residue field of its generic point. We say that $C$ is local, if 
$C$ is the spectrum of a local ring $R$.

\begin{definition} Let $G_K$ be a smooth connected  algebraic  group over $K$. An \emph{integral 
$C$-model} of $G_K$ 
is a smooth  flat group scheme $C$-scheme $\bfG$ with general fiber isomorphic to 
$G_K$. It is called a \emph{minimal integral model} if for 
any integral model $\bfG'$ of $G_K$, any morphism 
$G_K'\to G_K$ extends to a morphism of group schemes $\bfG'\to \bfG$.
\end{definition}

\begin{proposition} An integral $C$-model $\bfG$ of smooth  affine commutative algebraic group $G_K$
 is a  
commutative affine  group scheme over $C$.
\end{proposition}

\begin{proof} To prove the assertions we may assume  that $C = \Spec(R)$, where $R$ is a discrete valuation ring. 
The affine property  is proven in \cite[Proposition 2.3.1]{Anantharaman}. 
The commutativity is proven in \cite[Theorem 1.2]{Weisfeiler}.
\end{proof}

\begin{example}\label{ex1} Let $\sfU_K$ be a quasi-elliptic unipotent group over 
$K$ and $\bar{\sfU}_K$ be its minimal regular compactification of 
arithmetic genus one.  It follows from the theory of relative minimal models of two dimensional schemes that 
there exists a smooth 
projective flat scheme $f:J\to C$ over $C$  such that its generic fiber $J_K$ is isomorphic to 
$\bar{\sfU}_K$ and 
any birational morphism to any such $C$-scheme is an isomorphism. The closure of the cusp $P_\infty$ of $\sfU_K$ is a 
regular closed subscheme $\frakC$ of $J$  such that the restriction of $f$ to $\frakC$ is an 
inseparable finite morphism of degree $p = 2,3$. 

Let 
$$J^\sharp = \{x\in J:f\ \text{is smooth at $x$}\}$$
In particular, $\sfU_K = J_K^\sharp$ and $\frakC\subset J\setminus J^\sharp$. 
It follows from \cite[1.5, Proposition 1]{Bosch} that $J^\sharp\to C$ is the N\'eron model of 
$\sfU_K$. 
\end{example} 
 
\begin{example} Let $\sfU$ be an $F$-wound unipotent group over $F$ given by a Russell equation \eqref{russell1}. 
Assume that $F$ is the field $K$ of fractions of a discrete valuation ring $R$ with local parameter $\pi$. Then, 
replacing $u$ by $u\pi^{p^s}$ and $v$ by $v\pi^{p^{sn}}$ with sufficiently large $s$, one may assume that $a_i\in R$. 
This gives an integral model of $\sfU$.
\end{example}

It follows from the previous example that an integral model is not unique.  
Assume $C$ is local and let $G_K$ be any affine algebraic group over $K$ and  $G\to \Spec(R)$ be its 
integral model over $C = \Spec(R).$ 
For any proper subgroup $H$ of the closed fiber $G_t$ of $G$ defined by an ideal $I$  in the coordinate ring 
$R[G]$ of $G$, we consider the \emph{blow-up} $G^H$ of 
$H$ defined to be the spectrum of the subalgebra $R[\frakm^{-1}I]$ of $K[G] = R[G]\otimes_RK$. 
It is an integral model of $G$ and any integral model of $G_K$ that is mapped to $G$ with the image of the 
closed fiber equal to $H$. Any integral model of $G$ is obtained as the composition of the blow-ups and their inverses 
\cite{Waterhouse}. 


 
A minimal model of an abelian variety is its N\'eron model. 

\begin{definition} Let $G_K$ be a smooth   group scheme of finite type over $K$. A 
\emph{N\'eron model} of $G_K$ is a $C$-model $\bfG$ of $G_K$ which is smooth and of finite type satisfying 
the following universality property (the N\'eronian property):

For each smooth $C$-scheme $Y$ and each $K$-morphism $\sfU:Y_K\to G_K$, there is a unique 
$C$-morphism $u:Y\to \bfG$ extending $\sfU$.
\end{definition}

It follows from \cite[1.2, Criterion 9]{Bosch} that it is enough 
to check the weak N\'eronian property: 
for each point $c\in C$ and an \'etale $\calO_{C,c}$-algebra $R'$ with field of fractions 
$K'$, the canonical map $G(R')\to G_K(K')$ is surjective.

\begin{proposition} Let $S$ be an excellent  Dedekind scheme and $G_K$ be 
a smooth  commutative algebraic group  over the field $K$ of rational functions of $S$.
Assume that $G_K$ admits a regular compactification over $K$ and does not 
contain non-trivial subgroups unirational over $K$. Then $G_K$ admits a N\'eron model over $S$.
If $S$ is a regular algebraic curve over a field $F$, then the latter condition is necessary for the 
existence of a N\'eron model of $G_K$.
\end{proposition}

This is \cite[10.3, Theorem 5]{Bosch}. It follows that any $K$-wound unipotent group that 
admits a regular compactification over $K$ admits a N\'eron model over $S$. It is conjectured 
that the condition for the existence of a regular compactification can be dropped.

\begin{corollary}  Let $X_K$ be a regular geometrically reduced and irreducible proper curve of arithmetic genus 
$g > 0$ over $K$. Then the 
N\'eron model over $C$ of its Jacobian $\Jac(X_K)$ exists. 
\end{corollary}

In fact, in an  earlier influential paper \cite{Raynaud} 
 Raynaud proves that the N\'eron model of $\Jac(X_K)$ represents the maximal separated 
quotient of the relative Picard functor $\calP_{X/S}$ by an abelian 
subsheaf $\calE$ with fibers at closed points isomorphic $(\bbZ^{r_t})_t$, where $r_t$
is the number of irreducible components of $X_t$. (see also  \cite[9.5, Theorem 4]{Bosch}).

\begin{theorem} The groups scheme $\bfU = J^\sharp$ from Example \ref{ex1} is a N\'eron model of 
the   quasi-elliptic unipotent group $\sfU_K$. 
\end{theorem}

\begin{proof} It suffices to assume that $C = \Spec(R)$ is local.  Each connected  component of an  \'etale  scheme $Y_K$ over $K$ is 
isomorphic to $\Spec(L)$, where 
$L$ is a separable field extension  of $K$. Since the residue field $\Bbbk$ of the unique closed point of $C$ is 
algebraically closed, the normalization $R'$ of $R$ in $L$ has the same residue field. Thus we can apply 
\cite[7.2, Theorem 1]{Bosch} and assume that $R$ is henselian. Then any separable irreducible 
$K$-point is isomorphic to $\Spec(K)$, i.e. defines a rational $K$-point of $\sfU_K$. Since 
$J$ is projective over $R$, it defines a section of $\Spec(R)\to J$. The usual properties of 
the intersection theory of Cartier divisors show that it 
intersects the closed fiber at a smooth point. Hence its defines a section of $J^\sharp$. This proves 
that $J^\sharp\to C$ is a N\'eron model of $\sfU_K$.
\end{proof}

\begin{remark} The group $G_K = \bbG_{a,K}$ admits a trivial integral model $\bbG_{a,C}$.  
However, 
it is not a  N\'eron model. In fact, assume $C = \Spec(R)$ is local with a local parameter $u$, 
the zero point of $G_K = \Spec(K[T])$ defined by $T= 1/u$ has schematic closure defined by the ideal 
$(uT-1)$ and hence does not extend to a $R$-point of $\bbG_{a,R}$. 
By Theorem 2.1 from loc.cit., any model of  $\bbG_{a,K}$ with connected fibers is 
isomorphic to $\bbG_{a,R}$.

The quasi-rational group $\sfU_K$  does not admit a N\'eron model over a local $C = \Spec(R)$. 
Any  local integral model over $R$ 
is given by $u^2+v+av^2 = 0$, where $a$ belongs to the maximal ideal $\frakm$ of $R$ \cite[Theorem 2.1]{Waterhouse}. 
Any $K$-rational point 
$(u,v) = (\frac{s}{1+as^2},\frac{s^2}{1+as^2}) \ne (0,0)$ extends to an integral point $(s_0s_1:s_1^2:s_0^2+as_1^2)$ 
of the projectivization $y^2+xz+ax^2 = 0$. However, it has the  point $(0:1:0)$ at infinity lying in the closed fiber.
So, it does not extends to a section of the open subset  $z\ne 0$ isomorphic to the integral model of $\sfU_K$.
\end{remark}

\begin{definition} A flat group scheme of finite type over a base scheme $S$  
is called \emph{unipotent} if its fibers 
are unipotent algebraic group schemes over the residue fields of the points.
\end{definition} 

Let us recall the relationship between the 
relative Picard sheaf $\calP_{X/C}$ (in \'etale topology) and the N\'eron model $\bfJ$ of $\Jac(X_K)$ (see \cite{Raynaud} and 
\cite[9.5]{Bosch}).

Let $\calP_{X/C}^0$ be the connected component of identity of $\calP_{X/C}$. It is a subsheaf 
of $\calP_{X/C}$ which consists of elements whose restriction to all fibers belongs to the identity 
component of the Picard scheme of the fiber. Although $\calP_{X/C}$ is not representable in general,  $\calP_{X/C}^0$ is representable by a separated scheme 
if $X\to C$ has a section. In fact, it coincides in this case with 
the identity  component $\bfJ^\circ$ of the N\'eron scheme of 
$\Jac(X_K)$. The N\'eron scheme $\bfJ$ coincides with the maximal separated quotient of the subscheaf 
$\calP_{X/C}'$ equal to the kernel of the total degree homomorphism $\deg:\calP_{X/C}\to (\bbZ)_C$.

\begin{theorem}\label{unipotency} Let $X_K$ be a regular completion of a unipotent group of genus $g> 0$ over $K$. 
Then the N\'eron model $\bfJ$ of $\Jac(X)$ is a unipotent group scheme over $C$. 
\end{theorem}

\begin{proof} We may assume that $C = \Spec(R)$, where $R$ is a complete discrete valuation ring. 
It suffices to prove that $\bfJ^\circ$ is a 
connected unipotent group.  Since 
 $\Jac(X_K)(K)\ne \emptyset$,  the set $X(C)$ is also non-empty, and hence, by above 
 $\bfJ^\circ$ represents $\calP_{X/C}^0$.  Since there is a section, the degree homomorphism $\deg$ is 
 surjective. We also have an exact sequence 
 $$0\to \calP_{X/C}^0\to \calP_{X/C}'\to (\bbZ^{r-1})_t\to 0,$$
 where $r$ is the number of irreducible components of the closed fiber $X_t$ and $(\bbZ^{r-1})_t$ is the 
 skyscraper sheaf associated to the group $\bbZ^{r-1}$. It follows that the 
 closed fiber $\bfJ_t^\circ$ of $\bfJ^\circ$ is isomorphic to a subscheme of $\bfpic_{X_t/\Bbbk}$ with 
 quotient a free abelian group.

 Let $r_t:\Pic(X) = \calP_{X/C}(C) \to \Pic(X_t)$ be the restriction homomorphism. It is known that 
 $r_t$ is surjective and its kernel is a group uniquely divisible by any integer $s$ prime to 
$p$ \cite[Proposition 2.1]{Artin}. It follows that the $s$-torsion subgroup 
${}_s\Pic(X_t)$ of $\Pic(X_t)$ is isomorphic to the $s$-torsion subgroup ${}_s\Pic(X)$ of $\Pic(X)$. 
Since ${}_s\Pic(X_t) = {}_s\bfpic_{X_t/\Bbbk}^0$, it also coincides with ${}_s\bfJ_t^\circ$ (where, as always, 
we identify a smooth algebraic group over an 
algebraically closed field $\Bbbk$ with its group of $\Bbbk$-points). 

Let $
r_K:\Pic(X)\to \Pic(X_K)$ be the restriction homomorphism to the generic fiber. Its kernel is a free abelian group 
$\bbZ^r$ generated by irreducible components of $X_t$. By Theorem \ref{thm:ptorsion}, $\Pic(X_K)^0 = \Jac(X_K)(K)$ 
is a $p^{n'}$-torsion group. Let $\Pic(X)^0 = r_K^{-1}(\Pic(X_K)^0)$. It follows that 
${}_s\Pic(X)^0 = {}_s\Pic(X) = {}_s\bfJ_t^\circ = \{0\}$.

The only connected commutative algebraic group over an algebraically closed field that has this property is a vector group. 
This proves that $\bfJ^\circ$ is a unipotent group scheme over $C$.

\end{proof}

Let $\calF$ be a quasi-coherent  sheaf of $\calO_S$-modules over a scheme $S$.  The affine 
scheme $\bbV(\calF) = \Spec(\Sym^\bullet\calF)$  has a structure 
of a commutative group scheme over $S$. Its value on any $S$-scheme $a:T\to S$ is the group of 
global sections 
$\Gamma(T,a^*\calF^\vee)$, where $\calF^\vee$ is the dual sheaf of $\calO_S$-modules.  It is smooth if and only if 
$\calF$ is a locally free sheaf.  In this case it is called a \emph{vector group scheme}. A homomorphism of vector 
group schemes $\bbV(\calE_1)\to \bbV(\calE_2)$ corresponds to homomorphisms of locally free sheaves 
$\calE_2\to \calE_1$ or $\calE_1^\vee\to \calE_2^\vee$. The closed embeddings correspond to surjective homomorphisms.
 
Let $G$ be a commutative group scheme of finite type over a base $S$. Recall that its identity component $G^\circ$ is the open group subscheme 
whose fibers are the identity components of the fibers of $G$. We set 
$$\pi_0(G) = G/G^\circ$$
and call it the group of connected components. Its fiber over a closed point $s\in S$ is the group 
$\pi_0(G_s)$ of connected components of $G_s$.
 
 We also recall the definition of the Lie algebra $\Lie(G)$
 \cite{Demazure}, \cite{LLR}. 
 For any affine scheme scheme $T = \Spec(A)$, let  $T_\epsilon = \Spec(A[\epsilon]/(\epsilon^2))$ 
 with canonical projection 
 $T_\epsilon\to T$ defined by $a\mapsto a$ and a canonical section $T\to T_\epsilon$ 
 defined by $a+b\epsilon \mapsto a$. For any group $S$-scheme $G$, let 
 $\mathcal{L}ie(G)$ be the sheaf of abelian groups defined on category of affine $S$-schemes  
 by $\mathcal{L}ie(G)(A) = \Ker(G(T_\epsilon)\to G(T))$. 
 If $\Spec(R)$ is affine and $G = \Spec(\calO[G])$ is 
 affine, we have 
 $\mathcal{L}ie(G)(A) = \Der_R(\calO[G],A)$. The sheaf $\mathcal{L}ie(G)$ has a natural structure of an  
 $\calO_S$-module defined by $r\cdot (a+b\epsilon) = a+rb\epsilon$. We denote this $\calO_S$-module by 
 $\Lie(G)$. If $e:S\to G$ is the zero section of $G$,  
 \beq\label{deflie}
 \Lie(G)^\vee \cong e^*\Omega_{G/S}^1. 
 \eeq 
  For example, for any locally free sheaf $\calE$ of $\calO_S$-modules, we have 
  $\Omega_{\bbV(\calE)/S}^1 \cong a^*\calE$, where $a:\bbV(\calE)\to S$ is the natural projection, 
  and hence
 $e^*(\Omega_{\bbV(\calE)/S}^1) = \calE$ so that
  $$\Lie(\bbV(\calE)) = \calE^\vee.$$
 The formation $G\to \Lie(G)$ is a covariant functor on the category of group schemes over $S$. It is left exact. 
 
 \begin{theorem}\label{weisfeiler} Let $C$ be a local or global base,  \
$\sfU_K$ a genus $g > 0$ unipotent group over $K$, and 
$\bfU$  its N\'eron model. Then  $\bfU^\circ$ is isomorphic to a closed subgroup of a smooth group 
scheme  $\bfV$ locally isomorphic to a vector group scheme (not necessary a vector group scheme). 
\end{theorem}
 
 \begin{proof} Localizing at a closed point $t\in C$, 
 we get a unipotent group of genus $g$  over the field of fractions $K_t$ of a discrete 
 valuation ring $R = \calO_{C,t}$. We use Russell's equation \eqref{russell1} for $\sfU_{K_t}$. 
 If $\pi$ is a local parameter in $R$, we replace 
$(u,v)$ with $(\pi^{p^s}u,\pi^{p^{sn}v})$ for sufficiently large integer $s$ and 
cancelling by $\pi^{p^{sn}}$, we may assume  
that 
all $a_i\in R$. It follows from \cite[Theorem 2.8]{Waterhouse} that all smooth integral connected  models of $U_{K_t}$ 
over $R$ are given by such an equation. In particular, the identity component $\bfU_R^\circ$ of the 
N\'eron model $\bfU_R$ is given by such an equation.  
By the N\'eronian property, it is a minimal integral 
model over $R$. The equation allows one to embed $\bfU_R^\circ$ into the vector group scheme $\bbG_{a,R}^2 = \bbV(\calO_R^{\oplus 2})$. 

Now assume that $C$ is global. Let $\frakV = (V_i)_{i\in I}$ be a finite open affine cover of $C$ such that the 
restriction $\bfU_i^\circ$ of 
$\bfU_C^\circ$ to each $V_i\in \frakV$ is given by a Russell equation with coefficients in $\calO_C(V_i)$ 
$$\Phi_i(u_i,v_i) = u_i^{p^n}+v_i+a_1^{(i)}v_i^p+\cdots+a_m^{(i)}v_i^{p^m} = 0.$$ 
Since any derivation of $\calO(V_i)[u_i,v_i]/(\Phi_i)$ is a derivation of 
$\calO(V_i)[u_i,v_i]$ vanishing on $v_i$, the $\calO(V_i)$-module $\Lie(\bfU_i) = \Lie(\bfU_i^\circ)$ is generated  by 
$\frac{\partial}{\partial u_i}$. 

For brevity of notation let  $\calL$ be the invertible sheaf on $C$ equal to $\Lie(\bfU_C)$. 
Let $(c_{ij})$ be the  transition functions of $\calL$, so that $u_i = c_{ij}^{-1}u_j$. 
 
 The transition functions for $\bfU_C$ from $(v_i,u_i)$ to $(v_j,u_j)$ must be 
 $p$-polynomials with coefficients in  $\calO_C(V_i\cap V_j)$.
Suppose $n \le m$. Then 
\begin{eqnarray}
v_i &=& c_{ij}^{-p^n}v_j,\\
u_i &= &c_{ij}^{-1}u_j+c_{ij}^{-p^n}\alpha_{ij}^{(1)}v_j+\cdots +c_{ij}^{-p^m}\alpha_{ij}^{(m)}v_j^{p^{m-n}}, 
\end{eqnarray}
and 
\beq\label{transition1}
a_k^{(i)} = c_{ij}^{p^{n+k}-p^n}a_k^{(j)}-(\alpha_{ij}^{(k)})^{-p^n},
\eeq
where $\alpha_{ij}^{(k)}= 0$ for $k< n$. If $n> m$, then  $\alpha_{ij}^{(k)}$ are all zeros.

We can view $(1,a_1^{(i)},\ldots,a_m^{(i)})$ as a section of a vector 
bundle $\calA$ of rank $m+1$ given as an extension
\beq\label{extA}
0\to \calO_C \to \calA\to \calL^{\otimes p^n-p^{n+1}}\oplus\cdots\oplus \calL^{\otimes p^n-p^{m+n}}\to 0
\eeq 
with transition functions inverse to the transition functions \eqref{transition1}. 

We will use that  the \'etale cohomology of a vector group scheme $\bbV(\calE)$ are isomorphic 
to the abelian group of the Zariski cohomology of $\calE^\vee$ \cite[Chapter III, Proposition 3.7]{Milne}.

Let $\bfV$ be the group scheme locally isomorphic to  $\bbG_a^2$ with 
transition functions defined in above.
The  group scheme $\bfV$  fits into an extension of commutative group schemes
\beq\label{ext1}
0\to \bbV(\calL^{\otimes -p^n}) \to \bfV \to \bbV(\calL^{\otimes -1})  \to 0.
\eeq
given by the projection $(u,v) \to u$. It induces an exact cohomology sequence
$$0\to H^0(C,\calL^{\otimes p^n}) \to \bfV(C) \to H^0(C,\calL)$$
$$\to H_{\et}^1(C,\calL^{\otimes p^n})\to H_{\et}^1(C,\bfV) \to H^1(C,\calL)\to 0.$$
In the case $n > m$, exact sequence \eqref{ext1} splits  and we get 
$$H_{\et}^i(C,\bfV) \cong H^i(C,\calL)\oplus H^i(C,\calL^{\otimes p^n}).$$

The local embedding of $\bfU_i^\circ$ into $\bbG_{a,V_i}^2$ glue together to obtain an exact sequence
\beq\label{exseqU} 
0\to \bfU^\circ \to \bfV \overset{\mu}{\to} \bbV(\calL^{\otimes -p^n}) \to 0.
\eeq
\end{proof}

\begin{remark}\label{splitting} Note that $\bfV$ is a vector group only if $m= n$. 
In this case the transition matrices 
 are 
$$\begin{pmatrix}c_{ij}^{p^n}&-\alpha_{ij}c_{ij}^p\\
0&c_{ij}\end{pmatrix}.$$
The vector group scheme $\bfV$ is equal to $\bbV(\calE^\vee)$, where $\calE$ fits in an extension
$$0\to \calL^{\otimes p^n} \to \calE \to \calL \to 0.$$
It always splits, i.e. we may assume that $\alpha_{ij} = 0$, if $H^1(C,\calL^{\otimes 1-p^n}) = 0$.
For example, this happens if $(1-p^n)\deg(\calL) > 2g(C)-2$. If $\calE$  splits, we take 
$\alpha_{ij}^{(k)} = 0$  and hence  $\calA$ also splits.
\end{remark}

 \section{Global integral models of a quasi-rational unipotent group} \label{S5}
Let $\sfU_K$ be a quasi-rational unipotent group over the field of rational functions $K$ of 
 a complete smooth algebraic curve $C$ over an algebraically closed field 
$\Bbbk$. Its regular compactification is isomorphic to $\bbP_K^1$. 
A regular relatively minimal model of $\bbP_K^1$ over $C$ is a minimal ruled surface 
$f:X\to C$, where $X = \bbP(\calE)$ is a projective line bundle. 
This shows that we can try to construct an integral affine connected model $\bfU$ of $\sfU_K$ over $C$ 
as an open subset of $X$ whose complement is equal to the closure $\frakC$ of the point $P_\infty$ on 
the generic fiber.
The intersection of $\frakC$ with each closed fiber consists of one point taken with multiplicity 2, 
i.e. $\frakC$ is an inseparable bisection of $X$. In fact, starting from an inseparable point $P_\infty$ of degree 
$2$, we take any minimal ruled  surface $f:X\to C$, throw away the closure of $P_\infty$ and the complement will be 
an integral affine connected model of $\sfU_K$. We call such a model an \emph{affine minimal ruled surface}.

Let us remind some general facts about minimal ruled surfaces $\bbP(\calE)$ \cite[Chapter V, \S 2]{Hartshorne}. 
 For any invertible sheaf $\calN$, we have 
$\bbP(\calE)\cong \bbP(\calE\otimes \calN)$. One can normalize $\calE$ by tensoring it with an 
appropriate $\calN$  such that $H^0(C,\calE)\ne \{0\}$ but, for any invertible sheaf $\calN$ 
 of negative degree on $C$, one has $H^0(C,\calE\otimes \calN) = \{0\}$. Assume that $\calE$ is normalized.
  The inclusion 
  $\calO_C\to \calE$ defined by a section is saturated subsheaf of $\calE$, i.e. the quotient sheaf 
  has no torsion, and defines an extension of locally free sheaves
 \beq\label{hartshorne1}
 0\to \calO_C \to \calE \to \calL_0\to 0. 
 \eeq
 Any surjection  $\alpha:\calE\to \calN$ onto an invertible sheaf $\calN$  defines a section 
 $s_\alpha:C \cong \Proj(\Sym^\bullet\calN)\to 
 X = \Proj(\Sym^\bullet\calE)$ and conversely any section arises in this way. The sheaf 
 $s_\alpha^*(\Ker(\alpha)\otimes \calN^{\otimes-1})$ is isomorphic to the conormal sheaf 
 $\calO_E(-E)$, where $E = s_\alpha(C) \subset X$.

 Let 
 $s_0:C\to \bbP(\calE)$ be the section corresponding to the surjection $\calE\to \calL_0$. We will often 
 identify it 
 with its image $E_0$ in $X$. It is called 
 the \emph{exceptional section}. We have 
\beq\label{selfintersection} 
\calO_{E_0}(E_0) \cong \calL_0.
\eeq
 In particular,
 $$e:= -E_0^2 = \deg(\calL_0) = \deg(\calE) .$$

 \begin{lemma}\label{excsection1} For any section of $s:C\to E\subset \bbP(\calE)$, 
 $$E^2 \ge E_0^2.$$
 If $E\ne E_0$, the equality happens only in the case where the projective bundle is trivial, i.e. 
 $\calE\cong \calO_C\oplus\calO_C$. 
 \end{lemma}
 
 \begin{proof}  
 Let $\alpha:\calE\to \calN$ be a surjection corresponding to $E$. Then either 
the restriction of $\alpha$ to the subsheaf $\calO_C$ of $\calE$ is the zero map, or it is an injection. 
 In the former case, $\alpha$ factors through
 an isomorphism $\calL_0\to \calN$ and hence defines the same section.  In the latter case, 
 $\deg\calN\ge 0$ and the kernel of $\alpha$ is mapped isomorphically onto $\calL_0$. Hence 
 the exact sequence \eqref{hartshorne1} splits and $\calE\cong \calO_C\oplus \calL_0$. We also have 
 $\deg\calE = \deg\calL_0+\deg\calN$ gives $\deg\calN = 0$ and, since it has a section, we get   
 $\calN\cong \calO_C$. The conormal bundle of the section $E$ defined by $\calN$ is 
 isomorphic to $\calL_0$, hence $E^2 = -E_0^2$. Since $E^2\ge E_0^2,$ we get $E_0^2=E^2 = 0$ and, and hence 
 $\calL_0 \cong \calO_C$. This proves that $\calE \cong \calO_C^{\oplus 2}$. 
 \end{proof}

\begin{proposition} Let $\sfU$ be a quasi-rational unipotent group and $P_\infty\in P_K^1$
 its boundary point. There exists a unique (up to isomorphism)  integral affine minimal 
 ruled model $f:\bbP(\calE) \to C$ 
 such that the zero section is the exceptional section, the closure of 
 $P_\infty$ is disjoint from the zero section, and the restriction of $f$ to it  
 is isomorphic to the Frobenius cover $C^{(2)}\to C$. 
\end{proposition} 

\begin{proof} 
 Recall that an elementary transformation $\textrm{elm}_x$ with center at a 
 closed point $x\in \bbP(\calE)$  consists of 
blowing up  the point $x$  followed by blowing down the proper transform of the fiber $X_t$ 
where $t = f(x)$.
 If $x$ does not lie on the special section $\sfO$, the image of $\sfO$ has self-intersection equal to 
$\sfO^2+1$. The new projective line bundle is isomorphic to $\bbP(\Ker(\phi_x))$, where 
$\phi_x$ is the surjective map from $\calE$ to the sky-scrapper sheaf $(\Bbbk)_t$ defined by the point $x$.
The normalized  locally free sheaf $\calE'$ is an extension
$$0\to \calO_C\to \calE' \to \calL_0\to 0,$$
with $\deg\calE' = \deg\calE+1$. Suppose $\frakC$ is not smooth. After a sequence of elementary 
transformation we arrive at a new minimal model $f':X'\to C$ of $\bar{\sfU}$ with smooth closure  $\frakC'$ of 
$P_\infty$.

Note that, if we localize $X$ at a closed point $t\in C$, i.e. replace $C$ with a local base 
 $R = \Spec(\calO_{C,t})$, we obtain an integral model of $\sfU_{Q(R)}$ isomorphic to the affine surface
 $u^2+v+av^2 = 0$ and the localized $X$ is its projective closure $y^2+zx+ax^2 = 0$ in $\bbP_R^1$. 
 The localized bisection $\frakC_c$ is the closed subscheme $V(z)$. It is smooth if and only if $a_s$ generates the maximal 
 ideal of $R$. So, in the global equation of $\bfU$ discussed  in the previous section, the differentials 
 $da_i$ of local coefficients $a_i$ glue together to a differential 
 $$da\in H^0(C,\calL^{\otimes -2}\otimes \Omega_{C/\Bbbk}^1)$$ 
 of a section of $\calL^{\otimes -2}$ which 
 does not vanish anywhere. In particular, we obtain  
 \beq\label{thetachar}
 \calL^{\otimes 2} = \Lie(\bfU)^{\otimes 2}  \cong \omega_{C/\Bbbk}.
 \eeq

We  know that the zero section $\sfO$ of $\bfU$ is disjoint from the curve at infinity $\frakC$. 
Suppose we have another section $E$ with this property. 
Since their restrictions to the generic fiber are linearly equivalent,
 we have 
\beq\label{twosections} 
\sfO \sim E_0+f^*(D)
\eeq
 for some divisor class $D$ on $C$ of some degree $m$. In particular, 
 $\sfO^2 = E_0^2+2m$ shows that $m \ge 0$. 
Intersecting both sides of \eqref{twosections} with $\frakC$, we obtain 
$0 = \frakC\cdot E_0+2m$, hence $m = 0$ and $\frakC\cdot E_0 = 0$. 
Thus   $E_0$ is also  
disjoint from $\frakC$. But, then $E_0$ and $\sfO$ are  two section of the group scheme $\bfU$, hence 
must differ by a translation automorphism. Changing the zero section, we  may assume that 
$$E_0 = \sfO, \quad \calL = \calL_0.$$ 
Applying Lemma \ref{excsection1}, we obtain 
\beq\label  {nosections}
\bfU(C) = \{0\},
\eeq
unless the fibration $X\cong  C\times \bbP^1\to C$ is trivial. 
However, in this case we have a base-point free pencil $|\sfO|$
  of sections linearly equivalent to $\sfO$, so they all intersect 
 the curve $\frakC$, a contradiction.

We can also find the divisor class of the curve $\frakC$. 
Since the restrictions of $\frakC$ and $2\sfO$ to the generic fiber are linearly equivalent,
we have 
$$
\frakC\sim 2\sfO+f^*(D)
$$
for some divisor $D$ on $C$. Since $\frakC$ and $\sfO$ are disjoint,
restricting to $\sfO$, we obtain
$$\calO_C(D) \cong \calL^{\otimes -2} \cong \omega_C^{\otimes -1}.$$
Thus
\beq\label{compare2}
\frakC \sim 2\sfO-f^*(K_C).
\eeq
Comparing it with the formula for the canonical class 
\cite[Chapter V, \S 2, Lemma 2.10]{Hartshorne}, 
\beq\label{canclassruled} 
\omega_X \cong \calO_X(-2\sfO)\otimes f^*(\omega_C\otimes \calL)
\eeq
we find
\beq\label{compare3}
\calO_X(\frakC) \cong \omega_X^{-1}\otimes f^*(\calL).
\eeq
 
It remains to prove the uniqueness statement.  Recall that the map $\frakC$ is isomorphic to the 
Frobenius map. We have an exact sequence
$$0\to \calO_C\to f_*\calO_{\frakC}\to \calN \to 0$$
for some invertible sheaf $\calN$. 
By applying the direct image functor to exact sequence
$$0\to \calO_X(-\frakC)\to \calO_X\to \calO_{\frakC}\to 0,$$
using \eqref{compare3} and the relative duality 
$f_*\omega_{X/C} \cong (R^1f_*\calO_X)^\vee = 0$ and 
$R^1f_*\omega_{X/C} \cong (f_*\calO_X)^\vee = \calO_X$, we obtain 
that 
$$\calN \cong \omega_{C/\Bbbk}\otimes \calL^{\otimes -1} \cong \calL^{\otimes -1}.$$
This shows that  the isomorphism class of the  invertible sheaf $\calL$ is uniquely 
determined by $C$. In fact
it is isomorphic to the sheaf $\calB_{C/\Bbbk}^1$ equal to the image of 
$d:\bfF_*\calO_C\to \bfF_*\omega_{C/\Bbbk}$. Thus $\calL$ and hence $\bbP(\calE)$ is uniquely 
determined by the base curve $C$.
\end{proof}

\begin{example}\label{ex:conic} Assume $C = \bbP^1$. A minimal ruled surface over $\bbP^1$ is 
isomorphic to one of the 
surfaces $\bfF_n = \bbP(\calO_{\bbP^1}\oplus \calO_{\bbP^1}(-n)), n\ge 0$. 
We have $\Lie(\bfU) = \calO_{\bbP^1}(-n)$ 
and $\deg\Lie(\bfU) = -n$. It follows from \eqref{compare2} that $\frakC\in |2n\frakf+2\frake|$, where 
$\frakf$ is the divisor class of a fiber and $\bfe$ is the class of the special 
section $\sfO$ with self-intersection 
$-n$. The canonical class of $\bfF_n$ is equal to $-(n+2)\frakf-2\frake$. The adjunction formula shows that 
$\deg \omega_{\frakC} = 2n-4$ that agrees with the formula for the canonical class of $\frakC$ from above.  
Since $\frakC$ is smooth, it is of genus $g(C) = 0$. The cover $\frakC\to C$ is just the 
Frobenius morphism. The canonical class formula shows that $n = 1$ and $X = \bfF_1$.

The surface $\bfF_1$ is  
obtained by blowing up one point $p_0$ in the projective plane. 
The projection $f:\bbP(\calE)\to C$ is given by the pencil of lines through the point $p_0$. Choose 
projective coordinates $(x_0:x_1:x_2)$
such that $p_0 = (1:0:0)$ and consider a conic $Q$ with equation $x_1x_2+x_0^2 = 0$. 
Then any line from the pencil is tangent to the conic, hence the pre-image of $Q$ in 
$\bfF_1$ defines an inseparable bisection and its complement is the N\'eron model of a
quasi-rational curve $\sfU_K$. Let $t$ be the parameter of the pencil of lines so 
that $x_2+tx_1 = 0$. Then the generic fiber of the pencil is isomorphic to 
$\Proj(\Bbbk(t)[t_0,t_1])$ and the  complement of its intersection with  
 the conic is equal to the  open subset $D_+(tx_1^2+x_0^2)$. 
It is equal to the affine spectrum of 
the homogeneous localization of the graded ring $\Bbbk(t)[t_0,t_1]$ with respect to 
$f = tx_1^2+x_0^2$. The ring is generated by 
$u = x_0x_1/f,  v = x_1^2/f, w= x_0^2/f $ with relations 
$tv+w= 1$, and $vw+u^2 = 0$. This gives the equation $u^2+v+tv^2 = 0$ 
of $\sfU_K$ that agrees with a general equation \eqref{ratunipgroup} of a 
quasi-rational unipotent group.

\end{example}

\begin{example} 
If $g(C) = 1$, then $\calL \cong \calO_C$ or $\calL^{\otimes 2}\cong \calO_C$. 
In the former case, taking cohomology in exact sequence 
$$0\to \calO_C\to \bfF_*\calO_C\to \calL^{\otimes -1}\to 0$$
we obtain that $\bfF:H^1(C,\calO_C)\to H^1(C,\calO_C)$ is the zero map, hence $C$ is supersingular.
The ruled surface is a well known elliptic ruled surface defined by a non-split exact sequence
$$0\to \calO_C\to \calE\to \calO_C\to 0$$
with the extension class $1\in H^1(C,\calO_C) \cong \Bbbk$. In the latter case, we have 
$$\calE \cong \calO_C\oplus \calL.$$
and $C$ is an ordinary elliptic curve.
If $g(C) > 1$, $\deg \calL = g-1\ge 0$ and tensoring by $\calL^{\otimes -1}$, we see 
that the exact sequence cannot split (otherwise $\calE$ is not normalized). 
\end{example}

\section{Torsors of unipotent groups of genus $g$: local case}\label{S6}
Recall that  isomorphism classes of separable  \emph{torsors} 
(=principal homogeneous spaces) of any algebraic  group  $G$ over  
a field $F$ form a group isomorphic 
to the pro-finite  Galois cohomology group
$H^1(\Gal(F^{\sep}/F),G(F^{\sep}))$, where $F^{\sep}$ is the separable closure of $F$. If we 
view $G$ as an abelian sheaf that it represents in \'etale topology, then  
$$H^1(\Gal(F^{\sep}/F),G(F^{\sep})) = H_{\et}^1(F,G).$$
In this way one can consider torsors of any group scheme over $K$ where we replace 
the \'etale topology with 
the fpqc-topology (flat topology for short). If $G$ is a smooth group scheme then the flat cohomology 
coincides with \'etale cohomology \cite[Chapter III, Theorem 3.9]{Milne}.

We have already mentioned  in Section 2  that the 
unipotent group $\bbG_{a,F}^r$ does not have non-trivial torsors. 
However an $F$-wound unipotent group may have.

Let $\sfU_K$ be given by a separable $p$-polynomial $\Phi(x_1,\ldots,x_{r+1})$. 
We have an exact sequence of abelian sheaves in \'etale topology
\beq\label{kumerexseq1}
0\to \sfU_K\to \bbG_{a,F}^{r+1} \overset{\Phi}{\to} \bbG_{a,F}\to 0.
\eeq
Passing to cohomology, we get an isomorphism of abelian groups 
$$H^1(F,\sfU_K) \cong \bbG_{a}(F)/\Phi(\bbG_a(F)^{r+1}) = F/\Phi(F^{\oplus r+1}).$$
Explicitly, if $\sfU_K$ is given by a $p$-polynomial $\Phi$, an equation of a torsor is 
$$\Phi(x_1,\ldots,x_{r+1})+f = 0,$$
where $f\in F$. The torsor is trivial 
(i.e. has a rational point in $F$) if and only if 
$f\in \Phi(F^{\oplus r+1})$.

In this section we will study a special case when $\sfU_K$ is a genus $g$ unipotent group  and 
$F = K$ is strictly local, i.e. $K= \Bbbk((t))$  is the quotient ring of $R = \Bbbk[[t]]$ and 
$\Bbbk$ is algebraically closed. We denote by $\nu:K\setminus \{0\}\to \bbZ$ the discrete 
valuation of $K$.

We start with the following.

\begin{proposition}\label{genus0} 
Let $\sfU$ be a quasi-rational unipotent group over a strictly local field $K$. 
Then any its torsor is trivial.
\end{proposition}

\begin{proof} Of course, this immediately follows from the fact that the field $K$ is a $C_1$-field 
\cite{SerreGalois}, however we give a proof that  serves as a warm-up for our study of cases $g > 0$. We choose a minimal model of $\sfU_K$ defined by  
equation 
$$u^2+v+av^2 = 0,\  a\in R\setminus R^2.$$
Replacing $u$ with $t^su+c$ and $v$ with $t^{2s}$, we may assume that 
$a = t$.

Let $f\in K$ represent a non-zero element from $K/\Phi(K\oplus K)$. 
By taking $\Phi((x,0))$, we may assume that 
$f$ is not a square. By Hensel's Lemma, the equation 
 $h(v) = v+tv^2 = f$ can be solved for any $f\in R$. Thus $f$ can be represented by a
 negative Laurent polynomial
 $f = \sum_{k=0}^n c_kt^{-2k-1}$. Since
 $$
 h(ct^{-i}) = ct^{-i}+c^2t^{-2i},
$$
we see that any $t^{-2k-1}$ belongs to $\im(\Phi)$. 
 Thus there are no non-trivial torsors.
 \end{proof}

If  $\sfU_K$ is  a quasi-elliptic  unipotent group over strictly local $K$ of characteristic 3 a set of 
unique representatives of $K/\Phi(K\oplus K)$ was found by W. Lang \cite[Theorem 2.1]{Lang}. 

\begin{proposition}\label{langmult} Assume $p = 3$. Let $X$ be a non-trivial torsor of a quasi-elliptic unipotent group  
$\sf\sfU_K$.
Then $X$ is isomorphic to an affine curve over $K$ 
  given by one of the following equations:
$$u^3+v+t^kv^3+t^{-k}q_n(t^{-3}),$$
where $k = 1,2,4,5$ and $q_n(T)$ is a polynomial of some degree $n$. 
\end{proposition}
 
The case  $p = 2$ turns out to be more complicated \cite[4.8]{CDL}. We have 

\begin{proposition} Assume $p = 2$. let $X$ be a non-trivial torsor of a quasi-elliptic 
unipotent group $\sf\sfU_K$. Then $X$ is isomorphic to an affine curve over $K$ given by one 
of the following equations. 
\begin{enumerate}
\item  $u^2+v+tv^4+ t^{-1}q_n(t^{-4}) = 0$,
\item  $u^2+v+t^3v^4+t^{-3}q_n(t^{-4}) = 0$, 
\item  $u^2+v+t^5v^4+t^{-5}q_n(t^{-4}) = 0$, 
\item  $u^2+v+t^{2s+1}\epsilon^2v^2+tv^4+t^{-1}q_n(t^{-4}) = 0$,
\item  $u^2+v+t^{2s+1}\epsilon^2v^2+t^{2k+1}v^4+t^{-5}q_n(t^{-4}) = 0, k = 1,2$,
\item $u^4+v+(\epsilon^4+t\epsilon_2^4+t^2\epsilon_3^4+t^3\epsilon_3^4)v^2+ t^{-1}q_n(t^{-4}) = 0$,
\item $u^4+v+tv^2 +t^{-2}q_n(t^{-4}) = 0$,
\item  $u^4+v+t^2(\epsilon^4+t\epsilon_2^4+t^2\epsilon_3^4+t^3\epsilon_3^4)v^2 +t^{-3}q_n(t^{-2}), $
\item $u^4+v+t^3v^2+ t^{-6}q_n(t^{-4})= 0$,
\end{enumerate}
where $q_n(T)$ is a polynomial of some degree $n$.

\end{proposition}

Here the representatives are determined uniquely from the equation of $\sfU_K$ only in cases 
(1),(2), (3), (7) and (9).

\begin{remark} Let $X_t$ be the closed fiber of a relatively minimal model of a non-trivial torsor over $R$. 
Considered as a Cartier divisor on regular scheme $X$, it is equal to $p\bar{X}_t$ for some divisor $\bar{X}_t$ contained 
in $X_t$. A fiber is called tame if $f:X\to C$ is cohomologically flat \cite{Raynaud}, or, equivalently, 
 the normal sheaf $\calO_{\bar{X}_t}(\bar{X}_t)$ is an element of order 
$p$ in $\Pic(\bar{X}_t)$. It is conjectured by W. Lang that $X_t$ is tame if and only $n = 0$. 
The number $n$ is conjecturally  related to the length torsion subgroup of $H^1(X,\calO_X)$
\end{remark}

 \section{Torsors of unipotent groups of genus $g$: global case} \label{S7}
In this section we assume that $C$ is a smooth complete curve over an 
 algebraically closed field of characteristic $p > 0$. Let $K = \Bbbk(C)$ be its field of rational 
 functions. Let 
 $\sfU_K$ be a genus $g$ unipotent group over $K$ considered an abelian sheaf in the \'etale 
 topology of $K$. 
 
 Let us first get rid of the case  $g = 0$. Let $X_K$ be a torsor of $U_K$. Since $X_K$ and $U_K$ become 
 isomorphic after a separable extension of $K$, a regular compactification $\bar{X}_K$ is a genus 
 0 curve and hence isomorphic to a conic over $K$. By Tsen's Theorem the 
 global field $K$ is a $C_1$-field, hence every separable torsor is trivial \cite{SerreGalois}. This also follows from 
 the fact that $\Br(K) = \{0\}$.

From now on we assume that $g > 0$. We denote by $\bfU$ its N\'eron model.  We also consider $\bfU$ as an abelian sheaf in the \'etale 
 topology of $C$.

 Let 
 $$\iota: \eta= \Spec(K) \hookrightarrow C$$
 be the inclusion morphism of the generic point $\eta$ of $C$. 
 The Leray spectral sequence for the morphism $\iota$ and the sheaf $\sfU_K$ gives an exact sequence
 \beq
 0\to H_{\et}^1(C,\iota_*\sfU_K)\to H^1(K,\sfU_K) \to H^0(C,R^1\iota_*\sfU_K) \to H^2(C,\iota_*\sfU_K)\to H^2(K,\sfU_K).
 \eeq
 By definition of the N\'eron model, the sheaf $\iota_*\sfU_K$ is represented by the N\'eron model $\bfU$ of $\sfU_K$. 
The fiber  $(R^1\iota_*\sfU)(t)$ of $R^1\iota_*\sfU$ at a closed point $t\in C$ is isomorphic to $H^1(K_t,\sfU_{K_t}).$ The homomorphism 
$$\loc_t:H^1(K,\sfU_K) \to H^0(C,R^1\iota_*\sfU) \to (R^1\iota_*\sfU)(t) = H^1(K_t,\sfU_{K_t})$$
assigns to an isomorphism class $[X]$ of a torsor $X$ of $\sfU_K$ the isomorphism class of the torsor 
$X_{K_t}$ of $\sfU_{K_t}$. Since each torsor $X$ splits over some separable extension $L$ of $K$, $\loc_t([X]) = 0$ for 
any point $t$ which is not ramified in $L/K$. Let $\hat{R}_t$ be the formal completion of $R_t$ and $\hat{K}_t$ be its 
field of fractions. We have an isomorphism
$$H^1(K_t,\sfU_{K_t}) \cong H^1(\hat{K}_t,\sfU_{\hat{K}t})$$
\cite{Elkik}, so we are in a situation from the previous section. Passing to cohomology in exact 
sequence \eqref{kumerexseq1}, we obtain $H^2(K,\sfU) = 0$ \cite{Milne}. So, we get an exact sequence 
\beq\label{oggsha}
 0\to H_{\et}^1(C,\iota_*\sfU)\to H^1(K,\sfU_K) \to \bigoplus_{t\in C}H^1(\hat{K}_t,\sfU_{\hat{K}_t}) \to H_{\et}^2(C,\iota_*\sfU)\to 0.
 \eeq
The group $\Sh(\sfU,K):= H_{\et}^1(C,\bfU)$ is an analog of the Tate-Shafarevich group of an abelian variety. The 
group $H_{\et}^2(C,\bfU)$ is the group of obstructions to define a global torsor in terms of local torsors.

\begin{theorem}\label{mordellweil} Let $C$ be a complete curve over $\Bbbk$,  $\sfU_K$  a genus 
$g$ unipotent group over $K$, and $X_K$  its regular compactification. 
Then  $\Jac(X_K)(K)$ is a finite group killed by $p^{n'}$, where $n'$ is the pre-height of $\sfU_K$ (see 
section 3).
\end{theorem}

\begin{proof} Let $f:X\to C$ be a regular relatively minimal model of $X_K$. Here $X$ is a 
smooth projective surface over $\Bbbk$. Its Picard variety $\bfpic_{X/\Bbbk}^0$ is an abelian variety.
Since $\Bbbk$ is algebraically closed we may identify any variety over $\Bbbk$ with its set of $\Bbbk$-points.
Let $\Pic^0(X) = \bfpic_{X/\Bbbk}^0(\Bbbk)$. It is a subgroup of  $\bfpic_{X/\Bbbk}(\Bbbk) = \Pic(X)$ and the quotient 
$\NS(X) = \Pic(X)/\Pic^0(X)$ is the N\'eron-Severi group of the surface $X$. It known to be a finitely generated
abelian group. The pull back morphism 
$$f^*:\Jac(C) = \bfpic_{C/\Bbbk}^0\to \bfpic_{X/\Bbbk}^0$$
is a  
 homomorphism of abelian varieties.  
 By the Poincar\'e Reducibility Theorem (see \cite{Mumford}, Chapter 4, \S 19), 
 there exists an abelian variety $A$ over $\Bbbk$  and an isogeny
$$
A\times \Jac(C) \,\to\, \bfpic_{X/\Bbbk}^0.
$$
Since $\Jac(X_K)$ is an affine algebraic group, the image of $A$ 
under the restriction morphism $r_K:\bfpic_{X/\Bbbk}^0\to \bfpic_{X_K/K}^0 = \Jac(X_K)$, 
is equal to zero. However, the kernel of $r_\eta$ is generated by $f^*\Jac(C)$ and the subgroup  of 
divisor classes of irreducible components of fibers of $f$. The latter group is 
finitely generated abelian group, and hence $A = 0$. 
 We have $f^*\Jac(X)\subset \Pic^0(X)$, and hence $\Jac(X_K)(K)$ is finite generated abelian group.
On other hand, by Theorem \ref{thm:ptorsion}, $\Jac(X_K)$ is killed by  $p^{n'}$. This proves the assertion.
\end{proof}

\begin{corollary} Let $\sfU_K$ be an unipotent group of genus $g > 0$ over 
the field $K$ of rational function of 
an irreducible algebraic curve  over an algebraically closed field. 
Then $\sfU(K)$ is an elementary $p$-group.
\end{corollary}

Let $f:X\to C$ be a regular  relatively minimal model of $X_K$ over $C$. 
By definition, it does not contain 
$(-1)$-curves in fibers. If $g(C) > 0$, it also does not contain other $(-1)$-curves because they cannot map 
surjectively onto $C$. This $X$ is a minimal smooth algebraic surface over $\Bbbk$. If $g = 1$, it is 
a quasi-elliptic surface of Kodaira dimension $\le 1$. If $g > 1$, it is a surface of general type. 

Let $\frakC$ be the closure of the boundary point $P_\infty$ in $X_K$. It is a purely inseparable cover of $C$ of 
degree $p^k$, where $p^k$ is  the minimal degree of the splitting field of $\sfU_K$. 
Let 
$$X^\sharp = \{x\in X:f\ \text{is smooth at $x$}\}$$
The generic fiber $X_K^\sharp$ is equal to $\sfU_K$. 
Since $X\setminus X^{\sharp}$ is closed, the closure 
$\frakC$ of $P_\infty$ is contained in the complement $X\setminus X^\sharp$. 
 
 Let $X_0^\sharp$ be the open subset of $X^\sharp$ 
obtained by throwing away the irreducible components of fibers that do not intersect the zero section $\sfO$. The restriction of 
$\alpha$ to $X_0^\sharp$ defines a quasi-finite birational morphism of smooth affine schemes 
$\alpha_0:X_0^\sharp\to \bfU^\circ$. 
Applying the Zariski Main Theorem, we infer that $u_0$ is an isomorphism. Thus we can consider $\bfU^\circ$ as an open 
subscheme of $X^\sharp$ whose complement consists of irreducible components of fibers that do not intersect $\sfO$.

We know that the Neron model $\bfJ$ of $\Jac(X_K)$ is a maximal separated quotient of $\calP_{X/C'}$ 
by a discrete subsheaf supported at the points $t$ with reducible fiber $X_t$.  In particular its Lie
algebra $\Lie(\bfU)$ is isomorphic to the Lie algebra functor of $\calP_{X/C}'$ that, in its turn, is isomorphic 
to $R^1f_*\calO_X$ \cite[Proposition 1.3]{LLR}. The closed embedding 
$\sfU_K\hookrightarrow \Jac(X_K)$ defines a morphism $\sfU\to \bfJ$, and hence an injective homomorphism of 
the Lie algebras sheaves 
\beq\label{lieinclusion} 
\Lie(\bfU) \hookrightarrow \Lie(\bfJ) \cong R^1f_*\calO_X.
\eeq

 \begin{proposition}\label{vanishing} Let $f:X\to C$ be a relatively minimal model of a 
unipotent group $\sfU_K$ of genus $g> 0$ over a smooth projective curve $C$ over $\Bbbk$.  
 Suppose that the Picard scheme of $X$ is reduced. Then 
$$H^0(C,\Lie(\bfU)) = \{0\}.$$ 
\end{proposition}

\begin{proof} 
Since the generic fiber is geometrically irreducible, $f_*\calO_X \cong \calO_C$. The Leray spectral sequence for $f:X\to C$ and the sheaf $\calO_X$  gives an exact sequence
\beq\label{leray1}
0\to H^1(C,\calO_C) \to H^1(X,\calO_X) \to H^0(C,R^1f_*\calO_X) \to H^2(C,\calO_C) = 0.
\eeq
The linear spaces $H^1(C,\calO_C)$ and $H^1(X,\calO_X)$ are identified with the Lie algebras of the 
Picard schemes $\bfpic_{C/\Bbbk}^0$ and $\bfpic_{X/\Bbbk}^0$. Although the former scheme is 
always reduced, the latter one may not be reduced if $p_g(X) > 0$. The map between these 
spaces corresponds to the pull-back morphism of abelian varieties 
$f^*:\bfpic_{C/\Bbbk}^0\to \bfpic_{X/\Bbbk}^0.$ 
As in the proof of Theorem \ref{mordellweil}, using the  Poincar\'e Reducibility Theorem, we show 
that $f^*$ must be an isogeny,   
hence $H^0(C,R^1f_*\calO_X) = 0$.

Now the assertion  follows from the fact that $\sfU^\circ \subset \bfJ$, hence 
$\Lie(\bfU) \subset \Lie(\bfJ)$.
\end{proof}

\begin{example} If $X$ is a quasi-hyperelliptic surface, i.e. quasi-elliptic 
surface with no reducible fibers, then $C$ is an elliptic curve and  $\deg \calL = \deg R^1f_*\calO_X = 0$. In fact 
$\calL^{\otimes 4} \cong \calO_C$ if $p = 2$ and $\calL^{\otimes 6} \cong \calO_C$ if $p = 3$
\cite{BM}. In particular, it may occur that $\calL\cong \calO_C$ in which case 
$H^0(C,\calL) = \Bbbk$.  

  Since 
$R^if_*\calO_X = 0, i > 0$, the  Leray spectral sequence degenerates at page $E_2$ and 
 gives an isomorphism
\beq\label{leray2}
H^1(C,R^1f_*\calO_X) \cong H^2(X,\calO_X),
\eeq
and  the equality
\beq\label{leray3}
\chi(X,\calO_X) = -\deg R^1f_*\calO_X.
\eeq

If $g > 1$, the surface $X$ is of general type and a recent work of Yi Gu shows 
that $\chi(X,\calO_X) > 0$ \cite{Gu1}, \cite{Gu2}  
(eliminating a few possible cases where it may be not true from  \cite{Nick}). Thus in this case, 
$\deg R^1f_*\calO_X < 0$. However, this does not imply that $\deg \calL < 0$ since 
$R^1f_*\calO_X$ is not necessary a semi-stable vector bundle. In fact, it may occur that 
$H^0(C,\calL) \ne \{0\}$  and is isomorphic the Lie algebra $H^0(X,\Theta_{X/\Bbbk})$ of regular 
  vector fields on $X$ \cite[Theorem 4.1]{Lang1}. This happens for some Raynaud surfaces which we 
  will discuss in the next Example.

\end{example}

\begin{example}\label{raynaud} Assume that  $\sfU_K$ of genus $g > 0$ has Russell equation
$u^{p^n}+v+av^{p^n} = 0$. We can compactify the group in the usual way by homogenizing the equation 
$y^{p^n}+z^{p^n-1}x+ax^{p^n} = 0$. In the open set $x\ne 0$, the affine equation  is
$Y^{p^n}+Z^{p^n-1}+a= 0$. Its singular points are the 
zeroes of the differential $\omega = z^{p^n-2}dz+da$. Thus singular point has coordinate $z$ equal to zero 
and  $da$ must vanish at this point. 
To get a smooth compactification $f:X\to C$ would require to blow up 
this point that will force the fiber passing through the singular point to be reducible. 

Suppose that  $f:X\to C$ has no reducible fibers. 
Formula \eqref{transition1} implies that $da$ is well defined as a section  
$$da\in H^0(C,\calL^{\otimes -p^n(p-1)}\otimes \omega_{C/\Bbbk})$$
that does not vanish anywhere.  In this case, we also see that the closure 
$\frakC$ of $P_\infty$ is locally given by equation $Z= 0$ and it has a smooth point in each fiber. 
Thus $\frakC$ is smooth and 
$f:\frakC\to C$ is the iterated Frobenius map $\bfF^n:C\to C$. We have
$$\calL^{p^n(p-1)}\cong \omega_{C/\Bbbk}.$$
It implies that 
$$\deg \calL = \frac{2g(C)-2}{p^n(p-1)}$$
which is positive if $g(C) > 1$. 

If $g = 1$, we see that $\calL^{\otimes 4}\cong\calO_C$ if $p = 2$ (resp. $\calL^{\otimes 6}\cong  \calO_C$ if $p = 3$)
and $C$ in an elliptic curve. This agrees with classification of quasi-elliptic surfaces with no reducible fibers from 
\cite{BM}. If $\calL \cong \calO_C$,  
$\bfpic_{X/\Bbbk}$ is 
not reduced and 
 $\dim_\Bbbk H^0(C,\calL) = 1.$ 

Assume $n= 1$. Tne injection 
$\calO_C\to \bfF_*\calO_C$ corresponds to the inclusion $\calO_C^{p}\to \calO_C$ upstairs. 
Taking the differential 
$d:\calO_C\to \calB_{C/\Bbbk}^1$ with values in the sheaf of locally exact regular $1$-forms on $C$, we obtain an 
exact  sequence
$$0\to \calO_C\to \bfF_*\calO_C\to \calB\to 0$$
where $\calB = \bfF_*\calB_{C/\Bbbk}^1$ is a locally free sheaf of rank $p-1$. The 
sheaf $\calL^{\otimes p}$ is equal to $\bfF_*\calL$ and the section $da$ defines an injection 
$\calL^{\otimes p-1} \to \calB$. The curve $C$ admitting an invertible subsheaf $\calN$ such that 
$\calN^{\otimes p} \cong \omega_{C/\Bbbk}$ is called a \emph{Tango curve} and one says that the sheaf 
$\calN$ defines a \emph{Tango structure} on $C$ \cite{Takeda}. Thus 
we see that $\calL^{\otimes p-1}$ defines a Tango structure on $C$. The surface $X\to C$ is an example of a 
\emph{Raynaud surface}. These surfaces provide counter-examples in characteristic $p > 0$ to the 
Kodaira Vanishing Theorem 
\cite{Raynaud1} as well as 
examples of surfaces of general type with of non-zero regular vector fields \cite{Lang1}. 
\end{example}
 
Recall that a smooth projective surface $X$ over $\Bbbk$ is called \emph{supersingular} in sense of Shioda if 
its second Betti number $b_2(X)$ computed in \'etale topology 
coincides with the Picard number $\rho(X)$, the 
rank of the N'eron-Severi group $\NS(X)$. It is known that a Shioda-supersingular surface is 
\emph{supersingular} in sense of Artin that means that the formal completion of its Brauer group $\Br(X)$ is 
isomorphic to the formal completion of the vector group $\bbG_a^{r}$. The number $r$ is equal to 
$\dim_\Bbbk H^2(X,\calO_X) = p_g(X)$. The converse is true for Artin-supersingular K3 surfaces 
(which conjecturally are all unirational)  but is not known for surfaces   of general type.

Since our surface $X$ is uniruled, it is Shioda-supersingular \cite[Theorem 7.3]{Liedtke}.
 
\begin{lemma} The torsion group $\Tors(\NS(X))$ of the N\'eron-Severi lattice of 
$X$ is a finite $p$-group. It is trivial if $g = 1$ and $\chi(X,\calO_X) > 0$.
\end{lemma}

\begin{proof} For any commutative 
 finite flat group scheme $G$ there is a  natural isomorphism
 $H_{\text{fl}}^1(X,G) \cong \Hom(G^D,\bfpic_{X/\Bbbk}),$
 where $H_{\text{fl}}$ stands for the flat cohomology and $G^D$ is the Cartier dual group scheme 
 \cite[Proposition 6.2.1]{Raynaud}. 
 A non-trivial invertible sheaf $\calL$ of finite order $n$ in $\Pic(X)$ defines an embedding of 
 the constant group scheme $(\bbZ/n\bbZ)_X$ into  $\bfpic_{X/C}$. It corresponds to an 
 element $\alpha$ in 
 $H_{\fl}^1(X,\bmu_{n,X})$, where $\bmu_{n,X}$ is the kernel of $x\mapsto x^n$ for the multiplicative group 
 scheme $\bbG_{m,X}$. In its turn, the cohomology class $\alpha$ defines an isomorphism class of 
 $\bmu_{n,X}$-torsor $Y_\alpha$ 
 on $X$ in the flat topology. Since each irreducible fiber $X_t$ of $f$ is homeomorphic to $\bbP^1$, 
 the \'etale cohomology of $\mu_{n,X_t}$ are trivial for $(n,p) = 1$. Thus the restriction of 
 $Y_\alpha$ to an open subset of $X$ containing the generic fiber $X_K$ is trivial. Hence the restriction 
 of $\calL$ to $X_K$ is trivial and therefore $\calL\cong \calO_X(D)$ for some divisor $D$ contained in 
 fibers of $f$. It is known that  a divisor class that intersects with zero  each irreducible 
 component of a closed fiber is linearly equivalent to a rational linear 
 combination of the divisor classes $X_t$. Since $f$ has a section,  no fiber is multiple, and hence 
 $D$ is an integer linear combination of the classes of fibers, hence comes as the pull-back of
 a divisor class in $\Pic^0(C)$. Such a class belongs to $\Pic^0(X)$. 
 
 To prove the  second assertion we use Riemann-Roch Theorem and Serre's Duality Theorem on $X$ 
 that gives that 
 $\omega_X\otimes \calL^{\otimes -1}$ has a section defined by an effective divisor 
 $D'\sim K_X-D$, where $\calL\cong \calO_X(D).$ The restriction of 
 $K_X$ and $D$ to each irreducible component of a fiber is of degree $0$, hence 
$D'$ is a linear combination of fibers and therefore $D'= f^*(d')$ for some effective divisor class on $C$. Since 
$K_J = f^*(k)$ for some divisor class on $C$, we obtain that $D = f^*(k-d)$ where $k-d$ is a torsion divisor 
class on $C$. Hence $D\in f^*\Jac(C)$ and thus it is algebraically equivalent to zero.
\end{proof}

 \begin{theorem}\label{brauer} Let $f:X\to C$ be a regular relatively minimal model of a regular 
 compactification $\bar{\sfU}$ of a unipotent group of genus $g > 0$ and $\bfJ$ be the N\'eron model of 
 $\Jac(X_K)$. Then there is an isomorphism of abelian groups 
 $$H_{\et}^1(C,\bfJ) \cong \Bbbk^{p_g(X)}\oplus \Tors(\NS(X)),$$
 where $p_g(X) = \dim_{\Bbbk}H^2(X,\calO_X)$.
  \end{theorem}
 
 \begin{proof} It is known that 
 $H^1(C,\bfJ)$ is isomorphic to the Brauer group $\Br(X)$ of $X$ that coincides with the cohomological Brauer group 
 $H_{\et}^2(X,\bbG_m)$ \cite{Grothendieck}. Since the general fiber of $f:X\to C$ is geometrically 
 rational curve, the surface $X$ is a unirational surface, hence its second Betti number 
 computed in $\ell$-adic cohomology coincides with the rank of the Picard group. 
 The computation of the Brauer group of a smooth 
 surface gives an isomorphism of abelian groups 
 $$\Br(X) \cong \Bbbk^{p_g(X)}\oplus \Tors(\NS(X))$$
 (see \cite[0.10]{CDL}). 

 \end{proof}

We use the notations from the proof of Theorem \ref{weisfeiler}. We denote by 
$\mu$ the restriction of the homomorphism of group $\bfV\to \bbV(\calL^{\otimes -p^n})$ to the subgroup 
$\bbV(\calL^{\otimes -p^n}).$ The homomorphism $\mu$ is surjective in \'etale topology. Let 
$$G = \Ker(\mu).$$

Now everything is ready to prove our main theorem. 

\begin{theorem}\label{shafarevich} Let $\bfU$ be the N\'eron model of a 
unipotent group $\sfU_K$ of genus $g > 0$. Assume $H^0(C,\Lie(\bfU)) = \{0\}$.
  Then
\begin{itemize}
\item[(i)] $\bfU^\circ(C) = H_{\et}^0(C,\bfU^\circ)\cong H_{\et}^0(C,G) = G(C)$.\\
\item[(ii)] There is an exact sequence 
$$0\to H_{\et}^1(C,G) \to H_{\et}^1(C,\bfU^\circ) \to H^1(C,\calL) \to 0.$$\\
\item[(iii)] There is an exact sequence
$$0\to \pi_0(\bfU)/\bfU(C) \to H_{\et}^1(C,\bfU^\circ)\to H_{\et}^1(C,\bfU) \to 0.$$\\
\item[(iii)] $H_{\et}^2(C,\bfU) \cong H_{\et}^2(C,G)$.
\end{itemize}
\end{theorem}

\begin{proof} By assumption,  $H^0(C,\calL) = 0$, thus $H_{\et}(C,\bfV) \cong H^1(C,\calL^{\otimes p^n})$, hence 
$\Ker(H_{\et}^0(C,\bfV) \overset{H^0(\mu)}{\longrightarrow} H^0(C,\calL^{\otimes p^n}))$ 
coincides with  $H_{\et}^0(C,G)$. 
Now the first assertion  follows immediately 
from exact sequence \eqref{exseqU}.  
  Taking cohomology in \eqref{exseqU}, we get an exact sequence
\beq\label{mainexseq}
0\to   
H_{\et}^1(C,\bfU^\circ)\to  H_{\et}^1(C,\bfV) \to   
H^1(C,\calL^{\otimes p^n})\to H_{\et}^2(C,\bfU^\circ) \to 0.
\eeq

Let us look at the homomorphism   $H_{\et}^1(C,\bfV) \to   
H^1(C,\calL^{\otimes p^n})$. The group $G$ is the subgroup of 
$\bbV(\calL^{\otimes -p^n})$ equal to the kernel of the projection 
$\bfU^\circ\to \bbV(\calL)$. The image of $H_{\et}^1(C,\bfU^\circ)$ in  $ H_{\et}^1(C,\bfV)$ is equal to an extension 
of $H^1(C,\calL)$ with kernel $H_{\et}^1(C,G)$. This is the kerne of $ H_{\et}^1(C,\bfV) \to   
H^1(C,\calL^{\otimes p^n})$. Thus we obtain an exact sequence
\beq\label{exseqmain1}
0\to H_{\et}^1(C,G) \to H_{\et}^1(C,\bfU^\circ) \to H^1(C,\calL) \to 0
\eeq  
and an isomorphism
$$H_{\et}^2(C,\bfU^\circ) \cong H_{\et}^2(C,G).$$
Next we relate $\bfU^\circ$ and $\bfU$ by means of 
the exact sequence coming from the definition of $\pi_0(\bfU)$
$$0\to \bfU^\circ\to \bfU\to \pi_0(\bfU) \to 0.$$
Here $\pi_0(\bfU) = \bfU/\bfU^\circ$ is supported at a finite 
set of points in $C$ and so it 
can be identifies with 
its group of sections. 
Applying cohomology, we find
$$H_{\et}^2(C,\bfU) = H_{\et}^2(C,\bfU^\circ) \cong H_{\et}^2(C,G),$$
as asserted. We also get an exact sequence
\beq\label{exseqH1}
0\to \pi_0(\bfU)/\bfU(C) \to H_{\et}^1(C,\bfU^\circ) \to H_{\et}^1(C,\bfU)\to 0.
\eeq
\end{proof}

\begin{remark} Suppose $H^0(C,\calL) \ne \{0\}$. This implies that 
$\{0\} \ne H^0(C,\calL^{\otimes p^n}))\subset H_{\et}^0(C,\bfV)$ but the kernel $\bfU^\circ(C)$ of the map 
$\alpha:H_{\et}^0(C,\bfV)\to H^0(C,\calL^{\otimes p^n})$ must be a finite group.  This happens, for example, if 
 $f:X\to C$ is a quasi-hyperelliptic surface and $\calL\cong \calO_C$. In this case 
 $$H_{\et}^0(C,\bfV) = H^i(C,\calL) = H^i(C,\calL^{\otimes p^n}) = \bbG_a(\Bbbk),\  i =0,1,$$ 
  and $\bfU(C)$ is a non-trivial finite group isomorphic to 
 $\Tors(\NS(X))$. We know that $\deg \calL > 0$ is possible, for example, for 
 Raynaud surfaces from Example \ref{raynaud}. \end{remark}

The following Lemma is taken from \cite[III, \S4, Lemma 4.13]{Milne} where one can also 
 find references to its proof.
    
 \begin{lemma}\label{katz} Let $V$ be a finite-dimensional linear space of dimension $d$ over $\Bbbk$ and 
 $\phi:V\to V$ be a $p^k$-linear map (i.e. $\phi(\lambda x) = \lambda^{p^k}\phi(x)$ for any $\lambda\in \Bbbk$).
Let $V = V_{\text{ss}}\oplus V_{\text{nil}}$, where $\phi$ is bijective on $V_{\text{ss}}$ and $\phi$ is nilpotent 
 on $V_{\text{nil}}$. Then $\phi-\id$ is surjective on $V$ and  
 the kernel of $\phi-\id$  is a vector space  over 
 $\bbF_{p^k}$ of dimension equal to $\dim V_{\text{ss}}$. 
 \end{lemma}
 
 Note that the dimension of $V_{\textrm{ss}}$ can be computed as follows. Choose a 
$\Bbbk$-basis 
$\underline{e}$ of $V$ and let $\phi(\underline{e}) = A\underline{e}$ for some matrix 
$A$. Then 
\beq\label{hassewitt}
\dim_\Bbbk V_{\textrm{ss}} = \text{rank}(A\cdot A^{(p)}\cdot \cdots\cdot A^{(p^{d-1})}),
\eeq
where $A^{(p^k)}$ denotes raising the entries of $A$ in $p^k$-th power and $d = \dim_\Bbbk V$. The 
number $\dim_\Bbbk V_{\textrm{ss}}$ is called the \emph{stable rank} of the matrix $A$.
 
\begin{corollary}\label{katz2} Let $u^{p^n}+v+a_1v^p+\cdots+a_mv^{p^m} = 0$ be the equation 
of the N\'eron model
$\bfU$ of a unipotent group $\sfU_K$ of genus $g > 0$. Suppose that $a_1 =\cdots = a_{m-1} = 0$. Let 
$$\alpha:\calL^{\otimes p^n} \to \calL^{\otimes p^n}$$ be the map given by 
$v\mapsto v+av^{p^m}$ and 
$$H^1(\alpha):H^1(C,\calL^{\otimes p^n}) \to H^1(C,\calL^{\otimes p^n})$$
be the corresponding map on cohomology. Then $H^1(\alpha)$ is surjective and its kernel is 
a vector space over $\bbF_{p^m}$ of dimension equal to the stable rank $r$  of the $p^m$-linear map 
 $\phi = H^1(\alpha)-\id$.

In particular, 
$$H^1(C,G) \cong (\bbZ/p^m\bbZ)^r, \quad H_{\et}^2(C,\bfU) = H^2(C,G) = 0.$$
\end{corollary}

 Assume $\Tors(\NS(X)) = \{0\}$. The images of the 
 divisor classes of a fiber of $f:X\to C$ and of a section $\sfO$ generate a sublattice of $\NS(X)$ isomorphic 
 to the integral hyperbolic plane $\sfU$. It splits as an orthogonal summand of $\NS(X)$. Let 
 $\NS^0(X)$ be its orthogonal complement. The image of the restriction homomorphism 
 $\NS^0(X) \to \Pic(X_K)$ is isomorphic to the group of sections $ \bfU(C)$  and its kernel 
 is the sublattice $\NS_{\fib}^0(X)$ generated by components of fibers not intersection $\sfO$. 
 Let us consider a chain  
 of lattices and the corresponding dual lattices 
 $$\NS_{\fib}^0(X)\subset \NS(X)^0 \subset \NS^0(X)^\vee 
 \subset \NS_{\fib}^0(X)^\vee.$$
 The discriminant group $\NS_{\fib}^0(X)^\vee/\NS_{\fib}^0(X)$ of the lattice $\NS_{\fib}^0(X)$ is isomorphic to the group 
 $\pi_0(\bfU)$ \cite[8.1.2]{Raynaud} and the discriminant group 
 $\NS^0(X)^\vee/\NS^0(X)$ of the lattice  $\NS^0(X)$ is isomorphic to the discriminant group 
$D(\NS(X))$ of $\NS(X)$. This gives us a chain of  finite abelian groups
$$
\bfU(C) \subset \bfU(C)'\subset \pi_0(\bfU) 
$$
with quotients  $\bfU(C)'/\bfU \cong D(\NS(X))$ and $\pi_0(\bfU)/\bfU(C)' \cong \bfU(C).$
 
 Comparing it with exact sequences \eqref{exseqmain1} and \eqref{exseqH1}, we dare to make the 
 following conjecture.
 
 \begin{conjecture} The intersection $H_{\et}^1(C,G)_0$ of the subgroups  
 $H_{\et}^1(C,G)$ and $\pi_0(\bfU)/\bfU(C)$ of 
 $H_{\et}^1(C,\bfU^\circ)$   from assertions (ii) and (iii)  of the Theorem splits the exact sequenc
  $$
  0\to D(\NS(X)) \to \pi_0(\bfU)/\bfU(C) \to \bfU(C) \to 0. 
  $$
 The group $H_{\et}^1(C,\bfU)$ is isomorphic to $H^1(C,\calL)$ and fits in an extension
 $$0\to H_{\et}^1(C,G)/H_{\et}^1(G,G)_0 \to H_{\et}^1(C,\bfU) \to H^1(C,\calL) \to 0$$
 \end{conjecture}
 
 Note that $H^1(C,\calL)$ is a vector $p$-torsion group that may contain a finite elementary $p$-group with quotient 
 isomorphic to $H^1(C,\calL)$.
 
 \begin{example} Assume $g = 1, C = \bbP^1$ and $\calL \cong \calO_{\bbP^1}(-k)$. The linear 
 space  
 $H^1(C,\calL^{\otimes p}) \cong H^1(C,\calO_{\bbP^1}(-pk))$ has a natural basis formed by 
 negative Laurent monomials $e_i = t_0^{-i}t_1^{pk-i}, i = 1,\ldots,pk-1$ \cite[III,\S 5]{Hartshorne}.  
 Assume $p = 3$ and let $u^{3}+v+a_{6k}v^3 = 0$  be the equation of $\bfU$, where 
 $a_{6k}\in H^0(C,\calL^{\otimes 6}) = H^0(C,\calO_{\bbP^1}(6k))$ is a binary form  of degree $6k$. Write 
 $a_{6k} = \sum_{i=0}^{6k}c_it_0^{i}t_1^{6k-i}$. Let $A = (c_{ij})$ be a matrix with entries defined by 
  $$\Bigl(\frac{a_{6k}(t_0,t_1)}{t_0^{3i}t_1^{9k-3i}}\Bigr)' = \sum_{j=1}^{3k-1}c_{ij}e_j, i = 1,\ldots,
  3k-1.$$
Here $()'$ means that we eliminate from the  Laurent polynomial all monomials $t_0^it_1^j$ with 
 non-negative $i$, or $j$.  One computes the entries $c_{ij}$ of $A$ and obtains that 
  $$A = \begin{pmatrix}c_{3-1}&c_{6-1}&\cdots&c_{3d-1}\\
c_{3-2}&c_{6-2}&\cdots&c_{3d-2}\\
\vdots&\vdots&\vdots&\vdots\\
c_{3-d}&c_{6-d}&\cdots&c_{3d-d}\end{pmatrix}.
$$ 
 where $c_j = 0, j < 0$ and $d = 3k-1$.  The group 
 $H^1(C,G)$ is isomorphic to $(\bbZ/3\bbZ)^{\oplus r}$, where $r$ is the stable  rank of  
 $A$ (see \eqref{hassewitt}).
  The matrix $A$ coincides with the \emph{Hasse-Witt matrix} 
 that computes the $p$-rank of the 
 hyperelliptic curve $H$ of genus $d$ given by equation $t_2^2+a_{6k}(t_0,t_1) = 0$ 
 (i.e. the maximal $r$ such that $(\bbZ/p\bbZ)^r$ embeds in its Jacobian). Of course, in 
 our case, the polynomial $a_{6k}(t_0,t_1)$ may degenerate and does not define any hyperelliptic curve. 
 Note that the projection $\pi:H\to \bbP^1$ is a separable double cover ramified over $V(a_6)$ that gives
 and exact sequence
 $$0\to \calO_{\bbP^1}\to \pi_*\calO_H\to \calO_{\bbP^1}(-3k) \to 0$$
 and an isomorphism $H^1(H,\calO_H) = H^1(\bbP^1,\pi_*\calO_H) \cong H^1(\bbP^1,\calO_{\bbP^1}(-3k)).$ The 
 matrix $A$  describes the action of the Frobenius on the basis $(e_1,\ldots,e_{3k-1})$ of $H^1(H,\calO_H)$.

Assume $k = 1$. The surface $X$ is a rational 
 quasi-elliptic surface with a section. The classification of such surfaces is known, 
in particular, the group 
 $\bfU(C)$ is known in each case (see \cite[4.9]{CDL} for exposition of this classification and the references to the 
 original results). 
The conjecture is checked in characteristic $p = 2,3$ by explicit computation of the group 
$H^1(C,G)$ 
(see \cite[4.8]{CDL}). Note that in this case $D(\NS(X)) = \{0\}$ and $H^1(C,\calL) = 0$, so 
that the Tate-Shafarevich group $H_{\et}^1(C,\bfU)$ is trivial with agreement  with Theorem \ref{brauer}.

Under the same assumption, but taking $\calL = \calO_{\bbP^1}(-2)$, we get a quasi-elliptic 
K3 surface $f:X\to \bbP^1$. In this case Theorem \ref{brauer} gives that 
$H_{\et}^1(C,\bfU) \cong H^1(C,\calL) \cong \bbG_a(\Bbbk).$

The following two examples were suggested to me by T. Katsura.  It is known that the Fermat quartic surface
$x^4+y^4+z^4+w^4 = 0$ in characteristic $3$ is a supersingular K3 surface with Artin invariant $\sigma$ 
equal to $1$. It admits a quasi-elliptic fibration with Weierstrass equation 
$y^2+x^3+t_0^2t_1^2(t_0^8+t_1^{8}) = 0$ \cite{Katsura}. The Russell equation of $\bfU^\circ$ is 
$$u^3+v+t_0^2t_1^2(t_0^8+t_1^{8})v^3 = 0.$$
The quasi-elliptic fibration has 10 reducible fibers of Kodaira's type IV with 
$\pi_0(\bfU)\cong (\bbZ/3\bbZ)^{\oplus 10}$ and the Mordell-Weil group 
$\bfU(C)$ is isomorphic to $(\bbZ/3\bbZ)^{\oplus 4}$. The discriminant  group $D(\NS(S))$ is isomorphic to 
$(\bbZ/3\bbZ)^{\oplus 2}$. We compute the Hasse-Witt matrix $A$  and find 
that $H^1(C,G)\cong (\bbZ/3\bbZ)^{\oplus 4} \cong \bfU(C).$ 

However, if we take the K3 surface  given by Weierstrass equation 
$$y^2+x^3+t_0^2t_1^{10}+t_0^5t_1^7+t_0^8t_1^4+t_0^{10}t_1^2 = 0,$$
we obtain that its  Mordell-Weil group $\bfU(C)$ is  an elementary $3$-group of rank $2$, 
the quasi-elliptic fibration contains 10 
reducible fibers of type $IV$. Thus its Artin invariant $\sigma$ is equal to $3$, so that $D(\NS(S))$ is 
an elementary $3$-group of 
rank $6$ and $\pi_0(\bfU)$ is an elementary $3$-group of rank $10$. Computing the Hasse-Witt matrix $A$ we 
find that its stable rank  equal to   $4$. Thus 
$\bfU(C) \cong (\bbZ/3\bbZ)^{\oplus 2}$ is isomorphic to 
a proper subgroup of $H_{\et}^1(C,G) \cong (\bbZ/3\bbZ)^{\oplus 4}$. Katsura finds an explicit isomorphism 
from a certain subgroup $H_{\et}^1(C,G)'$ of $H_{\et}^1(C,G)$ to $\bfU(C)$.

\end{example}


\end{document}